\documentclass[reqno,10pt]{amsart}

\usepackage{amsmath} 
\usepackage{pdfsync}
\usepackage{parskip}
\usepackage{hyperref}
\usepackage{mdframed}
\usepackage{tikz}
\numberwithin{equation}{section}
\newtheorem{theorem}{Theorem}[section]
\newtheorem{prop}[theorem]{Proposition}

\newtheorem{lem}[theorem]{Lemma}
\newtheorem{cor}[theorem]{Corollary}

\theoremstyle{remark}
\newtheorem{remark}[theorem]{Remark}
\newtheorem{rmk}[theorem]{Remarks}

\def\R{\mathbb{R}}
\def\H{\mathbb{H}}
\def\L{\mathcal{L}}
\def\Lis{\mathcal{L}{\rm{is}}}
\def\id{{\rm{id}}}
\def\im{{\rm{im}}}
\def\supp{{\rm{supp}}}

\def\uf{{\rm{unif}}}

\def\x{x_{\mathsf{p}}}
\def\ev{{\mathsf{C}}}
\def\M{\mathsf{M}}
\def\B{\mathbb{B}(0,r)}
\def\U{\mathsf{O}}
\def\Uk{\mathsf{O}_{\kappa}}
\def\Uki{\mathsf{O}_{\iota}}
\def\Ukp{\mathsf{O}_{\kappa_{\sf p}}}
\def\Q{\mathbb{B}^{m}}
\def\Qk{\mathbb{B}^{m}_{\kappa}}
\def\Qi{\mathbb{B}^{m}}
\def\pk{\pi_\kappa}

\def\F{\mathfrak{F}}
\def\N{\mathbb{N}}
\def\tm{\theta^{\ast}_{\mu}}
\def\ttm{{\Theta}^{\ast}_{\mu}}

\def\ttl{{\Theta}^{\ast}_{\lambda,\mu}}
\def\tu{{T}_{\mu}}
\def\kb{\varphi^{\ast}_{\iota}}
\def\kf{\psi^{\ast}_{\iota}}
\def\kbk{\varphi^{\ast}_{\kappa}}
\def\kfk{\psi^{\ast}_{\kappa}}

\def\rh{{\varrho}^{\ast}_{\lambda}}
\def\tu{{T}_{\mu}}
\def\ez{\mathbb{E}_{0}}
\def\ef{\mathbb{E}_{1}}

\def\K{\mathbb{K}}
\def\f{\boldsymbol{f}}
\def\Re{\mathcal{R}}
\def\Rek{\mathcal{R}_{\kappa}}
\def\Rc{\mathcal{R}^c}
\def\Rck{\mathcal{R}^c_{\kappa}}

\def\Dfi{\rm{Diff}^{\hspace{.1em}\infty}}
\def\X{\mathbb{X}}
\def\Xk{\mathbb{X}_{\kappa}}
\def\Xo{\tilde{\varsigma}}
\def\Xt{\varsigma}

\def\Hom{{\text{Hom}}}

\begin{document}

\title[Continuous Maximal Regularity on Riemannian Manifolds]{Continuous Maximal Regularity on Uniformly Regular Riemannian Manifolds}
\author[Y. Shao]{Yuanzhen Shao}
\address{Department of Mathematics,
         Vanderbilt University, 
         Nashville, TN 37240, USA}
\email{yuanzhen.shao@vanderbilt.edu}
\author[G. Simonett]{Gieri Simonett}
\address{Department of Mathematics\\
         Vanderbilt University \\
         Nashville, TN~37240, USA}
\email{gieri.simonett@vanderbilt.edu}

\date{}
\thanks{
This work was partially supported by a grant from the Simons Foundation (\#245959 to G.~Simonett) and by
a grant from NSF (DMS-1265579 to G.~Simonett). }
\maketitle
\begin{abstract}
We establish continuous maximal regularity results for parabolic differential operators
acting on sections of tensor bundles on uniformly regular Riemannian manifolds $\M$.
As an application, we show that solutions to the Yamabe flow on $\M$ instantaneously regularize and become real analytic in space and time.
The regularity result is obtained by introducing a family of parameter-dependent diffeomorphims
acting on functions on $\M$ in conjunction with maximal regularity and the implicit function theorem.
\end{abstract}

\maketitle

\section{\bf Introduction}
\medskip

It is the main purpose of this paper to introduce a basic theory of continuous maximal regularity 
for parabolic differential operators acting on sections of tensor bundles
on a class of Riemannian manifolds called \emph{uniformly regular Riemannian manifolds} in this paper.  

This concept was first introduced by H. Amann in \cite{Ama13}. These manifolds may be non-compact, or even non-complete.
As a special case, any complete manifold $(\M,g)$ without boundary and with bounded geometry 
(i.e. $\M$ has positive injectivity radius and all covariant derivatives of its curvature tensor are bounded) 
is a uniformly regular Riemannian manifold, see \cite[Example~2.1(f)]{Ama13}.
The class of uniformly regular Riemannian manifolds 
is large enough in the sense that it satisfies most of the geometric conditions 
imposed by other authors in the study of geometric evolution problems.
In this paper and a subsequent one \cite{ShaoPre}, we mainly focus on two kinds of geometric evolution problems, namely the evolution of metrics and the evolution of surfaces driven by their curvatures. Numerous results have been formulated for geometric evolution equations over compact closed manifolds, which are special cases of uniformly regular Riemannian manifolds. 
Nowadays, there is increased interest in generalizing these results for non-compact manifolds, or manifolds with boundary.
Most of the achievements in this research line are formulated for complete manifolds with certain restrictions on their curvatures. Within the class of uniformly regular Riemannian manifolds, we are able to relax many of these constraints. One typical instance 
 is the Yamabe flow, a well-known geometric evolution problem. 
We will show in Section~5 that this problem possesses a local solution as long as the manifold has a compatible uniformly regular structure. 
Under appropriate assumptions on the background metric $g_0$, we will in addition show that
solutions instantaneously regularize and become real analytic in space and time. 

For a densely embedded Banach couple $E_1\overset{d}{\hookrightarrow}E_0$, we consider the following abstract linear inhomogeneous parabolic equation:
\begin{align}
\label{S1: abstract-eq}
\begin{cases}
\partial_t u(t) +A u(t)=f(t),\hspace*{1em}t\in\dot{I},\\
u(0)=u_0
\end{cases}
\end{align}
on $I:=[0,T]$. Given $I\in\{[0,T],[0,T)\}$, $\dot{I}:=I\setminus\{0\}$. $(E_0,E_1)$ is then called a pair of continuous maximal regularity for 
$A\in\mathcal{L}(E_1,E_0)$ if 
\begin{equation}
\label{MR}
(\partial_t +A,\gamma_0)\in{\Lis}(\ef(I),\ez(I)\times E_{\gamma}),
\end{equation}
that is, $(\partial_t +A,\gamma_0)$ is a linear isomorphism between the indicated spaces,
where $\gamma_0u:=u(0)$, $E_\gamma:=(E_0,E_1)_{\gamma,\infty}^0$, and $(\cdot,\cdot)_{\gamma,\infty}^0$ denotes the 
continuous interpolation method, see Section 2.2. 
The reader may refer to \eqref{S3: ez&ef} for the definitions of $\ez(I)$ and $\ef(I)$. 
\eqref{MR} implies that  problem \eqref{S1: abstract-eq} admits for each 
$(f,u_0)\in \ez(I)\times E_{\gamma}$ a unique solution $u$ which
has the best possible regularity, i.e. there is no loss of regularity as
$\partial_tu$ and $Au$ have the same regularity as $f$.
The theory of continuous maximal regularity provides a general and flexible tool
for the analysis of non-linear parabolic equations, including fully nonlinear problems.
 In some cases it can be viewed as a substitute for the well-known Nash-Moser iteration approach
to fully nonlinear parabolic equations. 
In addition to providing existence and uniqueness of solutions, 
maximal regularity theory combined with the implicit function theorem 
renders a powerful tool to establishing further regularity results for solutions of
nonlinear parabolic problems, and studying geometric properties of the semiflows
generated.
Here we refer to \cite{CleSim01, DaPra88, EscMaySim98, EscSim98, EscPruSim032} for a list of related work.  
The theory of maximal regularity has been well-formulated in Euclidean spaces, and used on compact closed manifolds.
However, not until recently was $L_p$-maximal regularity theory established on non-compact manifolds \cite{Ama14}, where the author uses retraction-coretraction systems of Sobolev spaces to translate the problem onto Euclidean spaces. 
We will make use of a similar building block to establish a theory of continuous maximal regularity on uniformly regular Riemannian manifolds without boundary. 
Here we should like to mention that
Amann's result \cite{Ama14} is of much greater generality and remarkably more technical, as 
maximal regularity results for parabolic boundary value problems on manifolds
with boundary are obtained.
Nevertheless, our result is distinctive in two respects. First, it is the first maximal regularity theory for H\"older continuous functions on uniformly regular manifolds. Secondly, our result is formulated for tensor-valued problems. This theory complements the work in \cite{EscMaySim98, EscSim98} for the compact case. Although we 
will not consider parabolic boundary value problems in this paper,
we nevertheless introduce function spaces on uniformly regular Riemannian manifolds
with boundary. This set-up may prove useful for further studies of geometric evolution 
equations, say the Yamabe flow, on manifolds with boundary. There are a few results dealing with (semilinear) parabolic equations on non-compact complete Riemannian manifolds under various curvature assumptions which are based on heat kernel estimates, e.g. Qi S. Zhang \cite{Zhang97, Zhang00}, A.L. Mazzucato and V. Nistor \cite{Mazz06}, F. Punzo \cite{Punzo12}, C. Bandle, F. Punzo and A. Tesei \cite{Band12}. The approach developed in \cite{Ama14} and in this paper does not rely on heat kernel estimates and is, thus, not limited to second order equations. It can be applied to a wide array of nonlinear parabolic equations, including quasilinear 
(and even fully nonlinear) equations.

A serious challenge in the development of a general and useful theory of function spaces on uniformly regular manifolds turns out to be the availability of interpolation results. 
This difficulty can be overcome by the interpolation result \cite[Section~1.18.1]{Trib78} in the case of Sobolev spaces, but is surprisingly difficult for H\"older spaces, which are natural candidates for the spaces $E_0, E_1$ in \eqref{S1: abstract-eq}. Thanks to the work of H. Amann in \cite{Ama13, Ama12}, we are able to build up the theory via a linear isomorphism $\f$ defined in Section~2.2.   

Section~2 is the step stone to the theory of continuous maximal regularity. Section~2.1 is of preparatory character, wherein we state the geometric assumptions and some basic concepts on tensor bundles and connections from manifold theory. In the first half of the subsequent subsection, we introduce H\"older continuous tensor fields and the corresponding retraction-coretraction theory of these spaces. This work, as aforementioned, was accomplished by H. Amann in his two consecutive papers \cite{Ama13, Ama12}. It paves the path for the interpolation, embedding, point-wise multiplication and differentiation theorems for H\"older continuous tensor fields in the second half of the same section and Section~2.3. 

The main theorem of this paper, Theorem \ref{generator of analytic semigroup-tensor}, is then formulated in Section~3. Its proof relies on the retraction-coretraction system established in Section~2.2 and a careful estimate of the lower order terms via interpolation theory. A resolvent estimate for so-called $(\mathcal{E},\vartheta;l)$-elliptic operators acting on tensor bundles is presented therein. Following the well-known semigroup theory and G. Da~Prato, P. Grisvard and S.~Angenent's work, we then prove a continuous maximal regularity theory on uniformly regular Riemannian manifolds. We will present the theory in a general format, that is to say, 
we establish maximal regularity for so-called {\em normally elliptic} differential
operators acting on tensor fields, with minimal regularity assumptions on the coefficients
of the differential operators.
One of the reasons for considering tensor fields, rather than scalar functions,
lies in the fact that this general result is used in a forthcoming paper \cite{ShaoPre}
to prove analyticity of solutions to the Ricci flow. 

In Section~4, we give a short introduction to a parameter-dependent translation technique on manifolds, which combined with maximal regularity theory serves as a very beneficial tool for establishing regularity of solutions to parabolic equations. The idea of employing a localized translation in conjunction with the implicit function theorem was initiated in \cite{EscPruSim031} by J. Escher, J. Pr\"uss and G. Simonett. Through the retraction-coretraction systems, we can thus introduce an analogy for functions over manifolds. We refer the reader to \cite{ShaoPre} for more information on this technique.

After importing all the theoretic tools, in Section~5, we thus can present an application 
to the Yamabe flow. The reader may refer to Section~5 for a brief historic account of this problem. It has been proved that the normalized Yamabe flow on compact manifolds admits a unique global and smooth solution, for smooth initial data, see \cite{Ye94}. We will show in this paper that this solution exists analytically for all positive times. Less is known about the Yamabe flow on non-compact manifolds. To the best of the authors' knowledge, all available results in this direction require the underlying manifold to be complete and have bounded curvatures, or to be of some explicit expression, see \cite{AnMa99} and \cite{Burch08}. We will formulate an existence and regularity result for the Yamabe flow on a manifold, which may not satisfy any of the above conditions.  

In the rest of this introductory section, we give the precise definition of uniformly regular Riemannian manifolds and present the existence of a localization system, which plays a key role in the retraction-coretraction theory established in Section~2.2. After that we briefly list some notations that we shall use throughout.
\smallskip\\
\goodbreak
{\textbf {Assumptions on manifolds:}}
In this section, we list some background information on manifolds, which provide the basis for the H\"older, little H\"older spaces and tensor fields on Riemannian manifolds to be introduced below. This fundamental work was first introduced in \cite{Ama13} and \cite{Ama12}. 
\smallskip\\
Let $({\M},g)$ be a $C^{\infty}$-Riemannian manifold of dimension $m$ with or without boundary endowed with $g$ as its Riemannian metric such that its underlying topological space is separable. An atlas $\mathfrak{A}:=({\Uk},\varphi_{\kappa})_{{\kappa}\in \mathfrak{K}}$ for ${\M}$ is said to be normalized if 
\begin{align*}
\varphi_{\kappa}({\Uk})=
\begin{cases}
\Q, \hspace*{1em}& \Uk\subset\mathring{\M},\\
\Q\cap\H^m, &\Uk\cap\partial\M \neq\emptyset,
\end{cases}
\end{align*}
where $\H^m$ is the closed half space $\R^{+}\times\R^{m-1}$ and $\Q$ is the unit Euclidean ball centered at the origin in ${\R}^m$. We put $\Qk:=\varphi_{\kappa}({\Uk})$ and  $\psi_{\kappa}:=\varphi_{\kappa}^{-1}$. 
\smallskip\\
The atlas $\mathfrak{A}$ is said to have \emph{finite multiplicity} if there exists $K\in{\N}$ such that any intersection of more than $K$ coordinate patches is empty. Put
\begin{align*}
\mathfrak{N}(\kappa):=\{\tilde{\kappa}\in\mathfrak{K}:\mathsf{O}_{\tilde{\kappa}}\cap\Uk\neq\emptyset \}.
\end{align*} 
The finite multiplicity of $\mathfrak{A}$ and the separability of $\M$ imply that $\mathfrak{A}$ is countable.
If two real-valued functions $f$ and $g$ are equivalent in the sense that $f/c\leq g\leq cf$ for some $c\geq 1$, then we write $f\sim g$.
\smallskip\\
An atlas $\mathfrak{A}$ is said to fulfil the \emph{uniformly shrinkable} condition, if it is normalized and there exists $r\in (0,1)$ such that $\{\psi_{\kappa}(r{\Qk}):\kappa\in\mathfrak{K}\}$ is a cover for ${\M}$. 
\smallskip\\
Following H. Amann \cite{Ama13, Ama12}, we say that $(\M,g)$ is a {\bf{uniformly regular Riemannian manifold}} if it admits an atlas $\mathfrak{A}$ such that
\begin{itemize}
\item[(R1)] $\mathfrak{A}$ is uniformly shrinkable and has finite multiplicity.
\item[(R2)] $\|\varphi_{\eta}\circ\psi_{\kappa}\|_{k,\infty}\leq{c(k)}$, $\kappa\in\mathfrak{K}$, $\eta\in\mathfrak{N}(\kappa)$, and $k\in{\N}_0$.
\item[(R3)] ${\psi_{\kappa}^{\ast}}g\sim{g_m}$, $\kappa\in\mathfrak{K}$. Here $g_m$ denotes the Euclidean metric on ${\R}^m$ and ${\psi_{\kappa}^{\ast}}g$ denotes the pull-back metric of $g$ by ${\psi_{\kappa}}$.
\item[(R4)] $\|{\psi_{\kappa}^{\ast}}g\|_{k,\infty}\leq c(k)$, $\kappa\in\mathfrak{K}$ and $k\in\N_0$.
\end{itemize}
Here $\|u\|_{k,\infty}:=\max_{|\alpha|\leq k}\|\partial^{\alpha}u\|_{\infty}$, and it is understood that a constant $c(k)$, like in (R2), depends only on $k$. An atlas $\mathfrak{A}$ satisfying (R1) and (R2) is called a \emph{uniformly regular atlas}. (R3) reads as
\begin{center}
$|\xi|^2/c\leq{\psi_{\kappa}^{\ast}g(x)(\xi,\xi)}\leq{c|\xi|^2}$,\hspace{.5em} for any $x\in{\Qk},\xi\in{\R}^m, \kappa\in\mathfrak{K}$ and some $c\geq{1}$.
\end{center}
We refer to \cite{Ama13b} for examples of uniformly regular Riemannian manifolds.

Given any Riemannian manifold $\M$ without boundary, by a result of R.E. Greene \cite{Greene78} 
there  exists a complete Riemannian metric $g_c$ with bounded geometry on $\M$,  see \cite[Theorem~2']{Greene78} and \cite[Remark~1.7]{MullerAx}. Hence we can always find a Riemannian metric $g_c$ making $(\M,g_c)$ uniformly regular. However, this result is of restricted interest, since in most of the PDE problems we are forced to work with a fixed background metric whose compatibility with the metric $g_c$ is unknown.
\smallskip\\
A uniformly regular Riemannian manifold $\M$ admits a \emph{localization system subordinate to} $\mathfrak{A}$, by which we mean a family $(\pi_{\kappa},\zeta_{\kappa})_{\kappa\in\mathfrak{K}}$ satisfying:
\begin{itemize}
\item[(L1)] ${\pk}\in\mathcal{D}({\Uk},[0,1])$ and $(\pi_{\kappa}^{2})_{\kappa\in{\mathfrak{K}}}$ is a partition of unity subordinate to $\mathfrak{A}$.
\item[(L2)] $\zeta_{\kappa}:={\kbk}\zeta$ with $\zeta\in\mathcal{D}({\Q},[0,1])$ satisfying $\zeta|_{\supp({{\kfk}\pi_{\kappa}})}\equiv 1$, $\kappa\in\mathfrak{K}$.
\item[(L3)] $\|\psi_{\kappa}^{\ast}{\pk}\|_{k,\infty} \leq{c(k)}$, for $\kappa\in\mathfrak{K}$, $k\in{\N}_0$.
\end{itemize}
The reader may refer to \cite[Lemma~3.2]{Ama13} for a proof. In addition to the above conditions, we will find it useful to define the following auxiliary function
\begin{align}
\label{section 1: biggest cut-off fn}
\varpi_{\kappa}:={\kbk}\varpi \text{ with } \varpi\in\mathcal{D}({\Q},[0,1]) \text{ satisfying that } \varpi|_{\supp(\zeta)}\equiv 1.
\end{align}
Lastly, if, in addition, the atlas $\mathfrak{A}$ and the metric $g$ are real analytic, we say that $(\M,g)$ is a {\bf{$C^{\omega}$-uniformly regular Riemannian manifold}}.

{\textbf{Notations:}}
Let ${\K}\in\{{\R},\mathbb{C}\}$. For any open subset $U\subseteq{\R}^m$, we abbreviate ${\F}^{s}(U,\mathbb{K})$ to ${\F}^{s}(U)$, where $s\geq 0$ and ${\F}\in \{bc,BC\}$. 
The precise definitions for these function spaces will be presented in Section 2. Similarly, 
${\F}^{s}({\M})$ stands for the corresponding $\mathbb{K}$-valued spaces defined on the manifold ${\M}$. 
\smallskip\\
Let $\| \cdot \|_{\infty}$ and $\| \cdot \|_{s,\infty}$ denote the usual norm of the Banach spaces $BC(U)$, $BC^{s}(U)$, respectively. Likewise, their counterparts defined on ${\M}$ are expressed by $\| \cdot \|_{\F}^{\M}$, where $\| \cdot \|_{\F}$ stands for either of the norms defined on $U$.
\smallskip\\
Given $\sigma,\tau\in\N_0$, 
$$T^{\sigma}_{\tau}{\M}:=T{\M}^{\otimes{\sigma}}\otimes{T^{\ast}{\M}^{\otimes{\tau}}}$$ 
is the $(\sigma,\tau)$-tensor bundle of $\M$, where $T{\M}$ and $T^{\ast}{\M}$ are the tangent and the cotangent bundle of ${\M}$, respectively.
We write $\mathcal{T}^{\sigma}_{\tau}{\M}$ for the $C^{\infty}({\M})$-module of all smooth sections of $T^{\sigma}_{\tau}\M$,
and $\Gamma(\M,T^{\sigma}_{\tau}{\M})$ for the set of all sections.

For abbreviation, we set $\mathbb{J}^{\sigma}:=\{1,2,\ldots,m\}^{\sigma}$, and $\mathbb{J}^{\tau}$ is defined alike. Given local coordinates $\varphi=\{x^1,\ldots,x^m\}$, $(i):=(i_1,\ldots,i_{\sigma})\in\mathbb{J}^{\sigma}$ and $(j):=(j_1,\ldots,j_{\tau})\in\mathbb{J}^{\tau}$, we set
\begin{align*}
\frac{\partial}{\partial{x}^{(i)}}:=\frac{\partial}{\partial{x^{i_1}}}\otimes\cdots\otimes\frac{\partial}{\partial{x^{i_{\sigma}}}}, \hspace*{.5em} \partial_{(i)}:=\partial_{i_{1}}\circ\cdots\circ\partial_{i_{\sigma}} \hspace*{.5em} dx^{(j)}:=dx^{j_1}\otimes{\cdots}\otimes{dx}^{j_{\tau}}
\end{align*}
with $\partial_{i}=\frac{\partial}{\partial{x^i}}$. The local representation of 
$a\in \Gamma(\M,T^{\sigma}_{\tau}{\M})$ with respect to these coordinates is given by 
\begin{align}
\label{local}
a=a^{(i)}_{(j)} \frac{\partial}{\partial{x}^{(i)}} \otimes dx^{(j)} 
\end{align}
with coefficients $a^{(i)}_{(j)}$ defined on $\Uk$.

For a topological set $U$, $\mathring{U}$ denotes its interior.
If $U$ consists of only one point, we define $\mathring{U}:=U$. 
For any two Banach spaces $X,Y$, $X\doteq Y$ means that they are equal in the sense of equivalent norms. The notation $\Lis(X,Y)$ stands for the set of all bounded linear isomorphisms from $X$ to $Y$.
\bigskip
\section{\bf Function Spaces on Uniformly Regular Riemannian Manifolds}
Most of the work in Section 2.1, and 2.2 is laid out in \cite{Ama13} and \cite{Ama12} for weighted functions and tensor fields defined on manifolds with ``singular ends" characterized by a ``singular function" $\rho\in C^{\infty}(\M,(0,\infty))$. Such manifolds are uniformly regular iff the singular datum satisfies $\rho\sim 1_{\M}$. Because of this, we will state some of the results therein without providing proofs below.
\subsection{\bf Tensor Bundles}
Let $\mathbb{A}$ be a countable index set. Suppose $E_{\alpha}$ is for each $\alpha\in\mathbb{A}$ a locally convex space. We endow $\prod_{\alpha}E_{\alpha}$ with the product topology, that is, the coarsest topology for which all projections $pr_{\beta}:\prod_{\alpha}E_{\alpha}\rightarrow{E_{\beta}},(e_{\alpha})_{\alpha}\mapsto{e_{\beta}}$ are continuous. By $\bigoplus_{\alpha}E_{\alpha}$ we mean the vector subspace of $\prod_{\alpha}E_{\alpha}$ consisting of all finitely supported elements, equipped with the inductive limit topology, that is, the finest locally convex topology for which all injections $E_{\beta}\rightarrow\bigoplus_{\alpha}E_{\alpha}$ are continuous.
\smallskip\\
We denote by $\nabla=\nabla_g$ the Levi-Civita connection on $T{\M}$. It has a unique extension over $\mathcal{T}^{\sigma}_{\tau}{\M}$ satisfying, for $X\in\mathcal{T}^1_0{\M}$,
\begin{itemize}
\item[(i)] $\nabla_{X}f=\langle{df,X}\rangle$, \hspace{1em}$f\in{C^{\infty}({\M})}$,
\item[(ii)] $\nabla_{X}(a\otimes{b})=\nabla_{X}a\otimes{b}+a\otimes{\nabla_{X}b}$, \hspace{1em}$a\in\mathcal{T}^{{\sigma}_1}_{{\tau}_1}{\M}$, $b\in\mathcal{T}^{{\sigma}_2}_{{\tau}_2}{\M}$,
\item[(iii)] $\nabla_{X}\langle{a,b}\rangle=\langle{\nabla_{X}a,b}\rangle+\langle{a,\nabla_{X}b}\rangle$, \hspace{1em}$a\in\mathcal{T}^{{\sigma}}_{{\tau}}{\M}$, $b\in\mathcal{T}^{{\tau}}_{{\sigma}}{\M}$,
\end{itemize}
where $\langle{\cdot,\cdot}\rangle:\mathcal{T}^{\sigma}_{\tau}{\M}\times{\mathcal{T}^{\tau}_{\sigma}{\M}}\rightarrow{C^{\infty}({\M})}$ is the extension of the fiber-wise defined duality pairing on ${\M}$, cf. \cite[Section 3]{Ama13}. Then the covariant (Levi-Civita) derivative is the linear map
\begin{center}
$\nabla: \mathcal{T}^{\sigma}_{\tau}{\M}\rightarrow{\mathcal{T}^{\sigma}_{\tau+1}{\M}}$, $a\mapsto{\nabla{a}}$
\end{center}
defined by
\begin{center}
$\langle{\nabla{a},b\otimes{X}}\rangle:=\langle{\nabla_{X}a,b}\rangle$, \hspace{1em}$b\in\mathcal{T}^{\tau}_{\sigma}{\M}$, $X\in\mathcal{T}^{1}_{0}{\M}$.
\end{center}
For $k\in{\N}_0$, we define
\begin{center}
$\nabla^k: \mathcal{T}^{\sigma}_{\tau}{\M}\rightarrow{\mathcal{T}^{\sigma}_{\tau+k}{\M}}$, $a\mapsto{\nabla^k{a}}$
\end{center}
by letting $\nabla^0{a}:=a$ and $\nabla^{k+1}a:=\nabla\circ\nabla^k{a}$.
We can also extend the Riemannian metric $(\cdot|\cdot)_g$ from the tangent bundle to any $(\sigma,\tau)$-tensor bundle $T^{\sigma}_{\tau}{\M}$ such that $(\cdot|\cdot)_g:T^{\sigma}_{\tau}{\M}\times{T^{\sigma}_{\tau}{\M}}\rightarrow{\K}$ by setting
\begin{center}
$(\cdot|\cdot)_g:(a_1,\ldots,a_{\sigma},b_1,\ldots,b_{\tau})\times(c_1,\ldots,c_{\sigma},d_1,\ldots,d_{\tau})\mapsto{\prod_{i=1}^{\sigma}(a_i|c_i)_g}{\prod_{i=1}^{\tau}(b_i|d_i)_{g^{\ast}}}$,
\end{center}
where $a_i,c_i\in{T{\M}}$, $b_i,d_i\in{T^{\ast}{\M}}$ and $(\cdot|\cdot)_{g^{\ast}}$ denotes the induced contravariant metric. In addition,
\begin{center}
$|\cdot|_g:\mathcal{T}^{\sigma}_{\tau}{\M}\rightarrow{C^{\infty}}({\M})$, $a\mapsto\sqrt{(a|a)_g}$
\end{center}
is called the (vector bundle) \emph{norm} induced by $g$.
\smallskip\\
We assume that $V$ is a $\K$-valued tensor bundle on $\M$ and $E$ is a $\K$-valued vector space, i.e.,
\begin{center}
$V=V^{\sigma}_{\tau}:=\{T^{\sigma}_{\tau}\M, (\cdot|\cdot)_g\}$,\hspace{.5em} and\hspace{.5em} $E=E^{\sigma}_{\tau}:=\{\K^{m^{\sigma}\times m^{\tau}},(\cdot|\cdot)\}$, 
\end{center}
for some $\sigma,\tau\in\N_0$. Here $(a|b):=$trace$(b^{\ast}a)$ with $b^{\ast}$ being the conjugate matrix of $b$. By setting $N=m^{\sigma+\tau}$ , we can identify $\F^s(\M,E)$ with $\F^s(\M)^N$. 
\smallskip\\
Throughout the rest of this paper, we always assume that 
\smallskip
\begin{mdframed}
\begin{itemize}
\item $({\M},g)$ is a uniformly regular Riemannian manifold.
\item $(\pi_{\kappa},\zeta_{\kappa})_{\kappa\in\mathfrak{K}}$ is a localization system subordinate to $\mathfrak{A}$.
\item $\sigma,\tau\in \N_0$, $V=V^{\sigma}_{\tau}:=\{T^{\sigma}_{\tau}\M, (\cdot|\cdot)_g\}$, $E=E^{\sigma}_{\tau}:=\{\K^{m^{\sigma}\times m^{\tau}},(\cdot|\cdot)\}$.
\end{itemize}
\end{mdframed}
\smallskip 
For $K\subset \M$, we put $\mathfrak{K}_{K}:=\{\kappa\in \mathfrak{K}: \Uk\cap K\neq\emptyset\}$. Then, given $\kappa\in\mathfrak{K}$,
\begin{align*}
\Xk:=
\begin{cases}
\R^m \hspace*{1em}\text{if }\kappa\in \mathfrak{K}\setminus \mathfrak{K}_{\partial\M},\\
\H^m \hspace*{1em}\text{otherwise,}
\end{cases}
\end{align*}
endowed with the Euclidean metric $g_m$.

Given $a\in \Gamma(\M,V)$ with local representation $\eqref{local}$
we define ${\kfk}a\in E$ by means of
$ {\kfk}a=[a^{(i)}_{(j)}]$,
where $[a^{(i)}_{(j)}]$ stands for the $(m^{\sigma}\times m^{\tau})$-matrix with entries $a^{(i)}_{(j)}$ in the $((i),(j))$ position, with $(i)$, $(j)$ arranged lexicographically.

 For the sake of brevity, we set $\boldsymbol{L}_{1,loc}(\X,E):=\prod_{\kappa}{L}_{1,loc}({\Xk},E)$. Then we introduce two linear maps for $\kappa\in\mathfrak{K}$:
\begin{center}
$\Rck:{L}_{1,loc}({\M},V)\rightarrow{L}_{1,loc}({\Xk},E)$, $u\mapsto{\psi_{\kappa}^{\ast}({\pk}u)}$,
\end{center}
and
\begin{center}
$\Rek:{L}_{1,loc}({\Xk},E)\rightarrow{L}_{1,loc}({\M},V)$, $v_{\kappa}\mapsto{\pk}{\kbk}v_{\kappa}$.
\end{center}
Here and in the following it is understood that a partially defined and compactly supported tensor field is automatically extended over the whole base manifold by identifying it to be zero  outside its original domain. Moreover, 
\begin{center}
$\Rc:{L}_{1,loc}({\M},V)\rightarrow\boldsymbol{L}_{1,loc}(\X,E)$, $u\mapsto{(\Rck u)_{\kappa}}$,
\end{center}
and
\begin{center}
$\Re:\boldsymbol{L}_{1,loc}(\X,E)\rightarrow{L}_{1,loc}({\M},V)$, $(v_{\kappa})_{\kappa}\mapsto{\sum_{\kappa}\Rek v_{\kappa}}$.
\end{center}
\medskip

\subsection{\bf H\"older and Little H\"older Spaces}
Before we study the H\"older and little H\"older spaces on uniformly regular Riemannian manifolds, we list some prerequisites for such spaces on $\X\in\{\R^m,\H^m\}$ from \cite{Ama12}. 

Throughout this subsection, we assume that $k\in{\N}_{0}$. For any given Banach space $F$, the Banach space $BC^{k}(\X,F)$ is defined by
\begin{align*}
BC^{k}(\X,F):=(\{u\in C^k(\X,F):\|u\|_{k,\infty}<\infty \},\|\cdot\|_{k,\infty}).
\end{align*} 
The closed linear subspace $BU\!C^k(\X,F)$ of $BC^{k}(\X,F)$ consists of all functions $u\in BC^{k}(\X,F)$ such that $\partial^{\alpha}u$ is uniformly continuous for all $|\alpha|\leq k$. Moreover,
\begin{align*}
BC^{\infty}(\X,F):=\bigcap_{k}BC^{k}(\X,F)=\bigcap_{k}BU\!C^{k}(\X,F).
\end{align*}
It is a Fr\'echet space equipped with the natural projective topology. 
\smallskip\\
For $0<s<1$, $0<\delta\leq\infty$ and $u\in F^{\X}$, the seminorm $[\cdot]^{\delta}_{s,\infty}$ is defined by
\begin{align*}
[u]^{\delta}_{s,\infty}:=\sup_{h\in(0,\delta)^m}\frac{\|u(\cdot+h)-u(\cdot)\|_{\infty}}{|h|^s}, \hspace*{1em}[\cdot]_{s,\infty}:=[\cdot]^{\infty}_{s,\infty}.
\end{align*}
Let $k<s<k+1$. The \textbf{H\"older} space $BC^{s}(\X,F)$ is defined as 
\begin{align*}
BC^{s}(\X,F):=(\{u\in BC^k(\X,F):\|u\|_{s,\infty}<\infty \},\|\cdot\|_{s,\infty}),
\end{align*}
where $\|u\|_{s,\infty}:=\|u\|_{k,\infty}+\max_{|\alpha|=k}[\partial^{\alpha} u]_{s-k,\infty}$.
\smallskip\\
The \textbf{little H\"older} space of order $s\geq 0$ is defined by
\begin{center}
$bc^s(\X,F):=$ the closure of $BC^{\infty}(\X,F)$ in $BC^{s}(\X,F)$.
\end{center}
By \cite[formula~(11.13), Corollary~11.2, Theorem~11.3]{Ama12}, we have 
\begin{align}
\label{section 2: bc^k=BUC^k}
bc^k(\X,F)=BU\!C^k(\X,F),
\end{align}
and for $k<s<k+1$
\begin{center}
$u\in BC^s(\X,F)$ belong to $bc^s(\X,F)$ iff $\lim\limits_{\delta\rightarrow 0}[\partial^{\alpha}u]^{\delta}_{s-[s],\infty}=0$, \hspace{1em} $|\alpha|=[s]$.
\end{center}
In the following context, let $F_{\kappa}$ be Banach spaces. Then we put $\boldsymbol{F}:=\prod_{\kappa}F_{\kappa}$. We denote by $l_{\infty}(\boldsymbol{F})$ the linear subspace of $\boldsymbol{F}$ consisting of all $\boldsymbol{x}=(x_{\kappa})$ such that
\begin{align*}
\|\boldsymbol{x}\|_{l_{\infty}(\boldsymbol{F})}:=\sup_{\kappa}\|x_{\kappa}\|_{F_{\kappa}}
\end{align*}
is finite. Then $l_{\infty}(\boldsymbol{F})$ is a Banach space with norm $\|\cdot\|_{l_{\infty}(\boldsymbol{E})}$.
\smallskip\\
For ${\F}\in \{bc,BC\}$, we put $\boldsymbol{\F}^s:=\prod_{\kappa}{\F}^s_{\kappa}$, where ${\F}^s_{\kappa}:={\F}^s(\Xk,E)$. Denote by 
\begin{align*}
l_{\infty,\uf}(\boldsymbol{bc}^k)
\end{align*}
the linear subspace of $l_{\infty}(\boldsymbol{BC}^k)$ of all $\boldsymbol{u}=(u_{\kappa})_{\kappa}$ such that $\partial^{\alpha}u_{\kappa}$ is uniformly continuous on $\Xk$ for $|\alpha|\leq k$, uniformly with respect to $\kappa\in\mathfrak{K}$. 
Similarly, for any $k<s<k+1$, we denote by 
\begin{align*}
l_{\infty,\uf}(\boldsymbol{bc}^s)
\end{align*}
the linear subspace of $l_{\infty,\uf}(\boldsymbol{bc}^k)$ of all $\boldsymbol{u}=(u_{\kappa})_{\kappa}$ such that 
\begin{align}
\label{S2: infnty,uf}
\lim\limits_{\delta\rightarrow 0}\max_{|\alpha|=k}[\partial^{\alpha}u_{\kappa}]^{\delta}_{s-k,\infty}=0, 
\end{align}
uniformly with respect to $\kappa\in\mathfrak{K}$. 
\begin{center}
$\f:\boldsymbol{F}^{\X}\rightarrow\prod_{\kappa} {F}_{\kappa}^{\X}$,\hspace{.5em} $u\mapsto \f(u):=(\text{\rm{pr}}_{\kappa}\circ u)_{\kappa}$
\end{center}
is a linear bijection. Set $E_{\kappa}:=E^{\sigma}_{\tau}$ and $\boldsymbol{E}:=\prod_{\kappa}E_{\kappa}$. Then \cite[Lemma~11.10,~11.11]{Ama12} tells us that for $s \geq 0$
\begin{align}
\label{section 2:Lis bc^s}
\f\in\Lis(bc^s(\X,l_{\infty}(\boldsymbol{E})),l_{\infty,\uf}(\boldsymbol{bc}^s)),
\end{align}
and for $s>0$, $s\notin\N$
\begin{align}
\label{section 2:Lis BC^s}
\f\in\Lis(BC^s(\X,l_{\infty}(\boldsymbol{E})),l_{\infty}(\boldsymbol{BC}^s)).
\end{align}

Now we are in a position to introduce the counterparts of these spaces on uniformly regular Riemannian manifolds. For $k\in{\N}_{0}$, we define
\begin{center}
$BC^{k}({\M},V):=(\{u\in{C^k({\M},V)}:\|u\|_{k,\infty}^{\M}<\infty\},\|\cdot\|_{k,\infty}^{\M})$,
\end{center}
where $\|u\|_{k,\infty}^{\M}:={\max}_{0\leq{i}\leq{k}}\||\nabla^{i}u|_{g}\|_{\infty}$.
\smallskip\\
We also set
\begin{align*}
BC^{\infty}({\M},V):=\bigcap_{k}BC^{k}({\M},V)
\end{align*}
endowed with the conventional projective topology. Then
\begin{center}
$bc^{k}(\M,V):=$ the closure of $BC^{\infty}$ in $BC^{k}$.
\end{center}
Let $k<s<k+1$. Now the {\textbf{H\"older} space} $BC^{s}(\M,V)$ is defined by
\begin{align*}
BC^{s}(\M,V):=(bc^{k}({\M},V),bc^{k+1}({\M},V))_{s-k,\infty}.
\end{align*}
Here $(\cdot,\cdot)_{\theta,\infty}$ is the real interpolation method, see \cite[Example I.2.4.1]{Ama95} and \cite[Definition~1.2.2]{Lunar95}. It is a Banach space by interpolation theory. For $s\geq 0$, we define the \textbf{little H\"older} spaces by 
\begin{center}
$bc^{s}({\M},V):=$ the closure of $BC^{\infty}({\M},V)$ in $BC^{s}({\M},V)$. 
\end{center}
\smallskip 
\begin{theorem}
\label{retraction of BCk}
Suppose $s\geq 0$. Then $\Re$ is a retraction 
\begin{equation*}
\text{from $l_{\infty}(\boldsymbol{BC}^{s})$ onto $BC^{s}({\M},V)$}, 
\end{equation*}
and
\begin{equation*}
\text{from $l_{\infty,\uf}(\boldsymbol{bc}^{s})$ onto $bc^{s}({\M},V)$.}
\end{equation*}
Moreover, $\Rc$ is a coretraction in both cases.
\end{theorem}
\begin{proof}
See \cite[Theorem~6.3]{Ama13} and \cite[Theorem~12.1, 12.3, formula~(12.2)]{Ama12}. Note that for $s\notin\N_0$, $BC^{s}({\M},V)$ and $bc^{s}({\M},V)$ coincide with the spaces $B^s_{\infty}({\M},V)$ and $b^s_{\infty}({\M},V)$ defined in \cite[Section~12]{Ama12}.
\end{proof}

In the following proposition, $(\cdot,\cdot)^0_{\theta,\infty}$ and $[\cdot,\cdot]_{\theta}$ are the continuous interpolation method and the complex interpolation method, respectively. See \cite[Example I.2.4.2, I.2.4.4]{Ama95} for definitions. 
\begin{prop}
\label{interpolation of LH spaces}
Suppose that $0<\theta<1$, $0\leq{s}_0<s_1$ and $s=(1-\theta)s_0+\theta{s_1}$ with $s\notin\N$. Then
\begin{enumerate}
\item[(a)] $(BC^{s_0}({\M},V),BC^{s_1}({\M},V))_{\theta,\infty}\doteq BC^{s}({\M},V)\\
\doteq [BC^{s_0}({\M},V),BC^{s_1}({\M},V)]_{\theta}$,\hspace{.5em} for\hspace{.5em} $s_0,s_1\notin\N_0$.
\item[(b)] $(bc^{s_0}({\M},V),bc^{s_1}({\M},V))_{\theta,\infty}\doteq BC^{s}({\M},V)$,\hspace{.5em} for\hspace{.5em} $s_0,s_1\in\N_0$.
\item[(c)] $(bc^{s_0}({\M},V),bc^{s_1}({\M},V))_{\theta,\infty}^0\doteq bc^{s}({\M},V)\doteq [bc^{s_0}({\M},V),bc^{s_1}({\M},V)]_{\theta}$,\hspace{.5em} for\hspace{.5em} $s_0,s_1\notin\N_0$.
\item[(d)] $(bc^{s_0}({\M},V),bc^{s_1}({\M},V))^0_{\theta,\infty}\doteq bc^{s}({\M},V)$,\hspace{.5em} for\hspace{.5em} $s_0,s_1\in\N_0$.
\end{enumerate}
\end{prop}
\begin{proof}
See \cite[Corollary~12.2~(iii), (iv) and Corollary 12.4]{Ama12}. 
\end{proof}

\begin{prop}
\label{equivalence of function spaces on manifolds}
Suppose that $u\in{\F}^{s}({\M},V)$ for ${\F}\in \{bc,BC\}$. Then
\begin{center}
$\|\cdot\|_{{\F}^{s}}^{\M}\sim\|\Rc(\cdot)\|_{l_{\infty}(\boldsymbol{\F}^{s})}=\sup\limits_{\kappa} \|( \psi_{\kappa}^{\ast}(\pi_{\kappa}\cdot))_{\kappa}\|_{{\F}^{s}_{\kappa}}$,
\end{center}
\end{prop}
\begin{proof}
Since $\Re$ and $\Rc$ are continuous, there exists a constant $C>0$ such that for any $u\in{\F}^{s}(\M,V)$ we have $\frac{1}{C}\|\Rc u\|_{l_{\infty}(\boldsymbol{\F}^{s})}\leq\|u\|_{{\F}^{s}}^{\M}=\|\Re\Rc u\|_{{\F}^{s}}^{\M}\leq{C}\|\Rc u\|_{l_{\infty}(\boldsymbol{\F}^{s})}$.
\end{proof}

For a given $\kappa\in\mathfrak{K}$ and any $\eta\in\mathfrak{N}(\kappa)$, we define $S_{\eta\kappa}: E^{\X_{\eta}}\rightarrow  E^{\Xk}$ by
\begin{align*}
S_{\eta\kappa}:u\mapsto \varpi{\kfk}\pi_{\eta}{\kfk}\varphi_{\eta}^{\ast}u.
\end{align*}
\goodbreak
\begin{lem}
\label{coordinates transf}
Suppose that $\F\in\{bc,BC\}$. Then
\begin{center}
$S_{\eta\kappa}\in\L(\F^s_{\eta},\F^s_{\kappa})$,\hspace{1em} and $\|S_{\eta\kappa}\|\leq c$,
\end{center}
for $\eta\in\mathfrak{N}(\kappa)$. Here the positive constant $c$ is independent of $\kappa$ and $\eta\in\mathfrak{N}(\kappa)$. The statement still holds true with $\varpi$ replaced by $\zeta$ in the definition of $S_{\eta\kappa}$, or $\pi_{\eta}$ being replaced by $\zeta_{\eta}$.
\end{lem}
\begin{proof}
The case that $s\in\N_0$ follows from the point-wise multiplication result on $\X$, the chain rule, (R2) and (L3). The gaps left can be filled in by interpolation theory.  
\end{proof}

An alternative of the spaces ${\F}^s({\M})$ can be defined as follows:
\begin{center}
$\breve{\F}^s({\M}):=(\{u\in{L_{1,loc}({\M})}:{\kfk}(\pi^2_{\kappa}u)\in{\F}^s\},\|\cdot\|_{\breve{\F}^s({\M})})$, 
\end{center}
where ${\F}\in \{bc,BC\}$ and $\|u\|_{\breve{\F}^s({\M})}:=\|{\kfk}(\pi^2_{\kappa}u)\|_{l_{\infty}(\boldsymbol{\F}^s)}$. 
\begin{prop}
\label{equivalence definition}
$\breve{\F}^s({\M})\doteq {\F}^s({\M})$.
\end{prop}
\begin{proof}
The proof is straightforward. Indeed,
\begin{align*}
\|u\|_{\breve{\F}^s({\M})}=\|{\kfk}(\pi^2_{\kappa}u)\|_{l_{\infty}(\boldsymbol{\F}^s)}\leq{M\|\Rck u\|_{l_{\infty}(\boldsymbol{\F}^s)}}\leq{M\|u\|_{{\F}^s}^{\M}}.
\end{align*}
To obtain the other direction, we adopt Lemma~\ref{coordinates transf} to compute
\begin{align*}
\|\Rck u\|_{{\F}^s_{\kappa}}&\leq\sum_{\eta\in\mathfrak{N}(\kappa)}\|{\kfk}\pi_{\kappa}\varpi{\kfk}({\zeta_{\eta}}\pi_{\eta}^2{u})\|_{{\F}^s_{\kappa}}\leq{M}\sum_{\eta\in\mathfrak{N}(\kappa)}\|\psi_{\eta}^{\ast}(\pi_{\eta}^2{u})\|_{{\F}^s_{\eta}}\leq{M}\|u\|_{\breve{\F}^s({\M})}.
\end{align*}
\end{proof}
\medskip

\subsection{\bf Basic Properties}
In this subsection, we will list some basic properties of the function spaces and tensor fields introduced in the previous subsections. These properties are well known to be enjoyed by their counterparts in Euclidean spaces or on domains with smooth boundary. By using the interpolation and retraction properties set up above, we can verify them on a uniformly regular Riemannian manifold $\M$. 
\begin{prop}
\label{embedding theory}
${\F}^{t}({\M},V)\overset{d}{\hookrightarrow} bc^{s}({\M},V)$, where $t>s\geq 0$ and $\F\in\{bc,BC\}$.
\end{prop}
\begin{proof}
This result is a direct consequence of interpolation theory and the dense embedding $BC^{\infty}(\M,V)\overset{d}{\hookrightarrow}bc^s(\M,V)$.
\end{proof}

Let $V_j=V^{\sigma_j}_{\tau_j}:=\{T^{\sigma_j}_{\tau_j}\M,(\cdot|\cdot)_g\}$ with $j=1,2,3$ be $\K$-valued tensor bundles on $\M$. By bundle multiplication from $V_1\times V_2$ into $V_3$, denoted by
\begin{center}
${\mathsf{m}}: V_1\times V_2\rightarrow V_3$,\hspace{1em} $(v_1,v_2)\mapsto {\mathsf{m}}(v_1,v_2)$,
\end{center}
we mean a smooth bounded section $\mathfrak{m}$ of $\Hom(V_1\otimes V_2,V_3)$, i.e., 
\begin{align}
\label{section 2: bundle multiplication}
\mathfrak{m}\in BC^{\infty}(\M, \text{Hom}(V_1\otimes V_2,V_3)), 
\end{align}
such that $\mathsf{m}(v_1,v_2):=\mathfrak{m}(v_1\otimes v_2)$. \eqref{section 2: bundle multiplication} implies that  for some $c>0$
\begin{center}
$|{\mathsf{m}}(v_1,v_2)|_g \leq c|v_1|_g |v_2|_g$,\hspace{1em} $v_i\in \Gamma(\M,V_i)$ with $i=1,2$.
\end{center}
Its point-wise extension from $\Gamma(\M,V_1\oplus V_2)$ into $\Gamma(\M,V_3)$ is defined by:
\begin{align*}
\mathsf{m}(v_1,v_2)(p):=\mathsf{m}(p)(v_1(p),v_2(p))
\end{align*}
for $v_i\in \Gamma(\M,V_i)$ and $p\in\M$. We still denote it by ${\mathsf{m}}$. We can also prove the following point-wise multiplier theorem for function spaces over uniformly regular Riemannian manifolds.

\begin{prop}
\label{pointwise multiplication properties}
Let $k\in{\N}_{0}$, and $V_j=V^{\sigma_j}_{\tau_j}:=\{T^{\sigma_j}_{\tau_j}\M,(\cdot|\cdot)_g\}$ with $j=1,2,3$ be tensor bundles. Suppose that $\mathsf{m}:V_1\times V_2\rightarrow V_3$ is a bundle multiplication. Then $[(v_1,v_2)\mapsto \mathsf{m}(v_1,v_2)]$ is a bilinear and continuous map for the following spaces:
\begin{center}
${\F}^{s}({\M},V_1)\times{\F}^{s}({\M},V_2)\rightarrow{\F}^{s}({\M},V_3)$, where $s\geq{0}$ and ${\F}\in\{bc,BC\}$.
\end{center}
\end{prop}
\begin{proof}
This assertion follows from \cite[Theorem~13.5]{Ama12}, wherein point-wise multiplication results for anisotropic function spaces are presented. For the reader's convenience, we will state herein a brief proof for 
the isotropic case.  
\smallskip\\
(i) Let $(M,g)=(\Xk,g_m)$. Set $E_j=E^{\sigma_j}_{\tau_j}:=\{\K^{m^{\sigma_j}\times m^{\tau_j}},(\cdot|\cdot)\}$ with $j=1,2,3$. Suppose $\mathsf{b}_e\in\L(E_1,E_2;E_3)$, the space of all continuous bilinear maps: $E_1\times E_2\rightarrow E_3$, and denote its point-wise extension by $\mathsf{m}_e$. The case $s\in\N$ follows from the product rule. For the same reason, it suffices to prove the case $0<s<1$. One may check that 
\begin{align}
\label{section 2: Holder/bundle}
[\mathsf{m}_e(v_1,v_2)]^{\delta}_{s,\infty}\leq c ([v_1]^{\delta}_{s,\infty}\|v_2\|_{\infty}+[v_2]^{\delta}_{s,\infty}\|v_1\|_{\infty})
\end{align}
for $v_i\in E^{\Xk}_i$. Now it is an immediate consequence of the continuity of $\mathsf{b}_e$ that
\begin{align*}
\mathsf{m}_e\in \L({\F}^s(\Xk,E_1),{\F}^s(\Xk,E_2);{\F}^s(\Xk,E_3)).
\end{align*}
(ii) Suppose that $\mathsf{b}_e\in {\F}^{s}(\Xk,\L(E_1,E_2;E_3))$ and denote its point-wise extension by $\mathsf{m}_e$, i.e.,
\begin{align*}
\mathsf{m}_e:E^{\Xk}_1 \times E^{\Xk}_2\rightarrow E^{\Xk}_3: \hspace*{.5em}(v_1,v_2)\mapsto \mathsf{b}_e(x)(v_1(x),v_2(x)).
\end{align*}
Consider the multiplications:
\begin{align*}
\mathsf{b}_1\in \L(\L(E_1,E_2;E_3), E_1; \L(E_2,E_3)),\hspace*{.5em} (f,v_1)\mapsto f(v_1,\cdot),
\end{align*}
and
\begin{align*}
\mathsf{b}_2\in \L(\L(E_2;E_3), E_2; E_3),\hspace*{.5em} (g,v_2)\mapsto g(v_2),
\end{align*}
where $v_i\in E_i$. Denote by $\mathsf{m}_i$ the point-wise extension of $\mathsf{b}_i$. Then by step (i), we deduce that
\begin{align*}
\mathsf{m}_1 \in \L({\F}^{s}(\Xk,\L(E_1,E_2;E_3)), {\F}^{s}(\Xk,E_1); {\F}^{s}(\Xk,\L(E_2,E_3))),
\end{align*}
and
\begin{align*}
\mathsf{m}_2 \in \L({\F}^{s}(\Xk,\L(E_2,E_3)), {\F}^{s}(\Xk,E_2); {\F}^{s}(\Xk,E_3)).
\end{align*}
Since $\mathsf{m}_e(v_1,v_2)=\mathsf{m}_2 ( \mathsf{m}_1(\mathsf{b}_e,v_1),v_2)$, it yields
\begin{align*}
\mathsf{m}_e\in \L({\F}^{s}(\Xk,E_1),{\F}^{s}(\Xk,E_2);{\F}^{s}(\Xk,E_3)).
\end{align*}
Moreover, the norm of $\mathsf{m}_e$ only relies on $\|\mathsf{b}_e\|_{s,\infty}$.
\smallskip\\
(iii) We define ${\mathsf{m}_{\kappa}}$ by
\begin{align*}
\mathsf{m}_{\kappa}(\xi_1,\xi_2):=\kfk(\zeta_{\kappa}\mathsf{m}({\kbk}\xi_1,{\kbk}\xi_2))
\end{align*} 
for $\xi_i\in E_i^{\Xk}$. Now it is a consequence of (L3), \eqref{section 2: bundle multiplication} and \cite[Lemma~3.1(iv)]{Ama13} that
\begin{align*}
\mathsf{m}_{\kappa}\in BC^{k}(\Xk,\L(E_1,E_2;E_3)),\hspace*{1em} \|\mathsf{m}_{\kappa}\|_{k,\infty}\leq c(k)
\end{align*}
for each $k\in\N_0$ and the constant $c(k)$ is independent of $\kappa$. Thus $\mathsf{m}_{\kappa}$ is a bundle multiplication. Now we conclude from (ii) that
\begin{align*}
\mathsf{m}_{\kappa}\in \L({\F}^{s}(\Xk,E_1),{\F}^{s}(\Xk,E_2);{\F}^{s}(\Xk,E_3)).
\end{align*}
Moreover, the norm of $\mathsf{m}_{\kappa}$ is independent of the choice of $\kappa$.
\smallskip\\
(iv) Given $v_1\in {\F}^{s}({\M},V_1)$ and $v_2\in {\F}^{s}({\M},V_2)$, we have
\begin{align*}
\Rck(\mathsf{m}(v_1,v_2))&={\kfk}(\pi_{\kappa}\mathsf{m}(v_1,v_2))\\
&=\sum_{\eta\in\mathfrak{N}(\kappa)}\mathsf{m}_{\kappa}(\Rck v_1,\varpi{\kfk}(\pi^2_{\eta}v_2)).
\end{align*}
The discussion in (iv) shows that 
\begin{align*}
\|\Rck(\mathsf{m}(v_1,v_2))\|_{t,\infty,\Xk}&\leq \sum_{\eta\in\mathfrak{N}(\kappa)}\| \mathsf{m}_{\kappa}(\Rck v_1,\varpi {\kfk}(\pi^2_{\eta}v_2))\|_{s,\infty,\Xk}\\
&\leq c\sum_{\eta\in\mathfrak{N}(\kappa)}\|\Rck v_1\|_{s,\infty,\Xk}\|\varpi{\kfk}(\pi^2_{\eta}v_2))\|_{s,\infty,\Xk}\\
& \leq c\sum_{\eta\in\mathfrak{N}(\kappa)}\|\Rck v_1\|_{s,\infty,\Xk}\|\mathcal{R}^{c}_{\eta}v_2\|_{s,\infty,\X_{\eta}}\\
& \leq cK \|\Rck v_1\|_{s,\infty,\Xk}\|v_2\|_{s,\infty}^{\M}
\end{align*}
The penultimate line follows from the point-wise multiplication result in $\Xk$ and Lemma \ref{coordinates transf}. The last line is a straightforward consequence of (R1). Thus Theorem \ref{retraction of BCk} implies that 
\begin{align*}
\mathsf{m}\in\L({\F}^{s}({\M},V_1),{\F}^{s}({\M},V_2);{BC^{s}({\M},V_3)}).
\end{align*}
By adopting a density argument based on Proposition \ref{embedding theory}, we can show that in fact 
\begin{align*}
\mathsf{m}\in\L(bc^{s}({\M},V_1),bc^{s}({\M},V_2);{bc^{s}({\M},V_3)}).
\end{align*}
Indeed, pick an arbitrary $t>s$. For each $v_i\in bc^{s}({\M},V_i)$ with $i=1,2$, there exist $(u_j^i)_j\in BC^{t}({\M},V_i)$ converging to $v_i$ in $BC^{s}({\M},V_i)$. Then by the above discussion and the triangle inequality, we deduce that $\mathsf{m}(u_j^1,u_j^2)\in BC^t(\M,V_3)$ and $\mathsf{m}(u_j^1,u_j^2)\rightarrow\mathsf{m}(v_1,v_2)$ in $BC^{s}({\M},V_3)$, which implies that $\mathsf{m}(v_1,v_2)\in bc^{s}({\M},V_3)$. 
\end{proof}

Let $s\geq{0}$ and $l\in{\N}_0$. A linear operator $\mathcal{A}:C^{\infty}({\M},V)\rightarrow \Gamma({\M},V)$ is called a linear differential operator of order $l$ on ${\M}$, if we can find $\boldsymbol{\mathfrak{a}}=(a^r)_r\in \prod_{r=0}^l \Gamma(\M, V^{\sigma+\tau+r}_{\tau+\sigma})$ such that
\begin{align}
\label{section 2: globally-defined diff-op}
\mathcal{A}=\mathcal{A}(\boldsymbol{\mathfrak{a}}):=\sum\limits_{r=0}^l \ev(a^r,\nabla^r \cdot).
\end{align}
Here the complete contraction 
\begin{align*}
\ev:\Gamma(\M, V^{\sigma+\tau+r}_{\tau+\sigma}\times V^{\sigma}_{\tau+r})\rightarrow \Gamma(\M, V^\sigma_\tau): (a,b)\mapsto \ev(a,b)
\end{align*}
is defined as follows. Let $(i_1),(i_2),(i_3)\in\mathbb{J}^{\sigma}$, $(j_1),(j_2),(j_3)\in\mathbb{J}^{\tau}$ and $(r_1),(r_2)\in\mathbb{J}^{r}$.
\begin{align*}
&\ev(a,b)(p):=\ev(a^{(i_3;j_1;r_1)}_{(j_3;i_1)} \frac{\partial}{\partial x^{(i_3)}}\otimes \frac{\partial}{\partial x^{(j_1)}}\otimes\frac{\partial}{\partial x^{(r_1)}}\otimes dx^{(j_3)}\otimes dx^{(i_1)},\\
&\quad  b^{(i_2)}_{(j_2;r_2)} \frac{\partial}{\partial x^{(i_2)}}\otimes dx^{(j_2)}\otimes dx^{(r_2)})(p)\\
&=a^{(i_3;j_1;r_1)}_{(j_3;i_1)} b^{(i_1)}_{(j_1;r_1)}\frac{\partial}{\partial x^{(i_3)}}\otimes dx^{(j_3)}(p),
\end{align*}
in every local chart and for $p\in\M$. The index $(i_2;j_1;r_1)$ is defined by
\begin{align*}
(i_3;j_1;r_1)=(i_{3,1},\cdots,i_{3,\sigma};j_{1,1},\cdots,j_{1,\tau};r_{1,1},\cdots,r_{1,r}).
\end{align*}
The other indices are defined in a similar way. \cite[Lemma~14.2]{Ama12} implies that $\ev$ is a bundle multiplication. Making use of \cite[formula~(3.18)]{Ama13}, one can check that for any $l$-th order linear differential operator so defined, in every local chart $({\Uk},\varphi_{\kappa})$ there exists some linear differential operator
\begin{align}
\label{section 2: local exp of diff-op}
\mathcal{A}_{\kappa}(x,\partial):=\sum\limits_{|\alpha|\leq{l}}a^{\kappa}_{\alpha}(x)\partial^{\alpha}, \hspace*{.5em}\text{ with }a^{\kappa}_{\alpha}\in \L(E)^{\Qk},
\end{align}
called the local representation of $\mathcal{A}$ in $(\Uk,\varphi_{\kappa})$, such that for any $u\in C^{\infty}({\M},V)$
\begin{align}
\label{section 2:local exp of diff-op 2}
{\kfk}(\mathcal{A}u)=\mathcal{A}_{\kappa}({\kfk}u).
\end{align}
What is more, the $\mathcal{A}_{\kappa}$'s satisfy 
\begin{align}
\label{section 2: well-defined diff-op}
{\psi_{\eta}^{\ast}}{\kbk}\mathcal{A}_{\kappa}{\kfk}{\varphi_{\eta}^{\ast}}=\mathcal{A}_{\eta}, \hspace*{1em}\eta\in\mathfrak{N}(\kappa).
\end{align}  
Conversely, suppose that $\mathcal{A}$ is a linear map acting on $C^{\infty}({\M},V)$, and its localizations satisfy \eqref{section 2: local exp of diff-op}-\eqref{section 2: well-defined diff-op}. Then $\mathcal{A}$ is a well-defined linear operator from $C^{\infty}({\M},V)$ to $\Gamma({\M},V)$. Moreover, one can show that $\mathcal{A}$ is actually a linear differential operator of order $l$ with expression \eqref{section 2: globally-defined diff-op}. Indeed, we can construct $\boldsymbol{\mathfrak{a}}=(a^r)_r$ in a recursive way as follows. For notational brevity, we express $\mathcal{A}_{\kappa}$ as
\begin{center}
$\mathcal{A}_{\kappa}= a^{(r)}_{\kappa}\partial_{(r)}$, \hspace{.5em}$(r)\in\mathbb{J}^s$ with $s\leq l$, and $a^{(r)}_{\kappa}\in \L(E)^{\Qk}$. 
\end{center}
Then we write the coefficients $a^{(r)}_{\kappa}$ of $\mathcal{A}^{\pi}_{\kappa}$, the principal part of $\mathcal{A}_{\kappa}$, in the matrix form $(a^{(i_2;j_1;r)}_{\kappa,(j_2;i_1)})$, where $a^{(i_2;j_1;r)}_{\kappa,(j_2;i_1)}$ is the entry of $a^{(r)}_{\kappa}$ in the $(i_1;j_1)$-th column and $(i_2;j_2)$-th row with indices ordered lexicographically. Here $(i_1),(i_2)\in\mathbb{J}^{\sigma}$, $(j_1),(j_2)\in\mathbb{J}^{\tau}$ and $(r)\in\mathbb{J}^{l}$. Then we define 
\begin{align}
\label{section 2: diff-op local-global}
a^l|_{\Uk}&:= a^{(i_2;j_1;r)}_{\kappa,(j_2;i_1)}\frac{\partial}{\partial x^{(i_2)}}
\otimes \frac{\partial}{\partial x^{(j_1)}}\otimes\frac{\partial}{\partial x^{(r)}}
\otimes dx^{(j_2)}\otimes dx^{(i_1)}\end{align}
on the local patch $\Uk$. It follows from \cite[formula~(3.18)]{Ama13} that 
\begin{align*}
{\kfk}\ev(a^l,\nabla^l u)=\mathcal{A}^{\pi}_{\kappa}{\kfk}u +\mathcal{B}_{\kappa}{\kfk}u,
\end{align*}
where $\mathcal{A}^{\pi}_{\kappa}$ is the principal part of $\mathcal{A}_{\kappa}$, and $\mathcal{B}_{\kappa}$ is a linear differential operator of order at most $l-1$ defined on $\Qk$. By the above argument, $\mathcal{A}$ is well-defined only if for all local coordinates $\varphi_{\kappa}=\{x_1,\cdots,x_m\}$ and $\varphi_{\eta}=\{y_1,\cdots,y_m\}$ with $\eta\in\mathfrak{N}(\kappa)$
\begin{align*}
{\psi_{\eta}^{\ast}}{\kbk}(\sum\limits_{(r)\in\mathbb{J}^l} a^{(r)}_{\kappa}\otimes {\kfk}\frac{\partial}{\partial x^{(r)}})=\sum\limits_{(r)\in\mathbb{J}^l} a^{(r)}_{\eta}\otimes {\psi_{\eta}^{\ast}}\frac{\partial}{\partial y^{(r)}}
\end{align*}
only if $a^l$ is invariant. Hence $a^l$ is globally well-defined. Therefore
\begin{align*}
\tilde{\mathcal{A}}u:=\mathcal{A}u-\ev(a^l,\nabla^l u)
\end{align*}
is a well-defined linear differential operator of order at most $l-1$. Then we repeat this process to lower the order of $\mathcal{A}$ till it reduces to zero. We find 
\begin{align}
\label{S2: diff-global-coef}
\boldsymbol{\mathfrak{a}}=(a^r)_r\in \prod_{r=0}^l \Gamma(\M, V^{\sigma+\tau+r}_{\tau+\sigma})
\end{align}
such that $\mathcal{A}$ can be formulated in the form of \eqref{section 2: globally-defined diff-op}. 
\smallskip\\
In the rest of this section, we assume that $\mathcal{A}$ is a linear differential operator of order $l$ over ${\M}$ with local expressions 
$
\mathcal{A}_{\kappa}(x,\partial):=\sum\limits_{|\alpha|\leq{l}}a^{\kappa}_{\alpha}(x)\partial^{\alpha}.$ 
We first state a useful proposition concerning the equivalence of regularity of local versus global coefficients.
\begin{prop}
\label{differential operator coef}
If $(a^{\kappa}_{\alpha})_{\kappa}\in{l_{b}(\boldsymbol{\F}^{s}({\Qk},\L(E))}$ for every $|\alpha|\leq l$ with $b=``\infty"$ for $\F=BC$, or $b=``\infty,\uf"$ for $\F=bc$, then $\boldsymbol{\mathfrak{a}}$ in \eqref{S2: diff-global-coef} belongs to the class $\prod_{r=0}^l \F^{s}(\M,V^{\sigma+\tau+r}_{\tau+\sigma})$, and vice versa.
\end{prop}
\begin{proof}
$(a^{\kappa}_{\alpha})_{\kappa}\in{l_{b}(\boldsymbol{\F}^{s}({\Qk},\L(E)))}$ implies that $(a^{(i_2;j_1;r)}_{\kappa,(j_2;i_1)})_{\kappa}\in{l_{b}(\boldsymbol{\F}^{s}({\Qk}))}$. By \eqref{section 2: diff-op local-global}, we attain $\Rck a^l={\kfk}\pi_{\kappa}a^{(r)}_{\kappa}\otimes {\kfk}\frac{\partial}{\partial x^{(r)}}$ with $(r)\in\mathbb{J}^{l}$. The point-wise multiplication results on $\Qk$ yield
\begin{align*}
\Rc a^l\in l_{b}(\boldsymbol{\F}^s(\Qk,E^{\sigma+\tau+r}_{\tau+\sigma})).
\end{align*}
By Theorem \ref{retraction of BCk}, we have that $a^l\in \F^s(\M,V^{\sigma+\tau+r}_{\tau+\sigma})$. The rest of the proof, including the converse statement, follows from the recursive construction of $a^r$ above and \cite[formula~(3.18)]{Ama13}.
\end{proof}
\begin{prop}
\label{differential operator properties 2}
Suppose that $\boldsymbol{\mathfrak{a}}=(a^r)_r\in \prod_{r=0}^l bc^{s}(\M,V^{\sigma+\tau+r}_{\tau+\sigma})$. Then 
\begin{center}
$\mathcal{A}\in\L({bc}^{s+l}(\M,V),{bc}^{s}(\M,V))$.
\end{center}
\end{prop}
\begin{proof}
By \cite[Theorem~16.1]{Ama12}, we have
\begin{align}
\label{section 2: gradient}
\nabla\in \L(bc^{s+1}(\M,V),bc^{s}(\M,V^{\sigma}_{\tau+1})),\hspace*{1em}s\notin\N_0.
\end{align}
The case that $s\in\N_0$ follows by the definition of H\"older spaces and a density argument. Since $\ev$ is a bundle multiplication, the statement is a straightforward conclusion of Proposition \ref{pointwise multiplication properties}. 
\end{proof}
The following corollary is a special case of Proposition \ref{differential operator properties 2}.
\begin{cor}
\label{differential operator properties}
Suppose that $\boldsymbol{\mathfrak{a}}=(a^r)_r\in \prod_{r=0}^l BC^{t}(\M,V^{\sigma+\tau+r}_{\tau+\sigma})$.  Then 
\begin{align*}
\mathcal{A}\in\mathcal{L}({\F}^{s+l}(\M,V),{\F}^{s}(\M,V)). 
\end{align*}
Here we choose $t>s$ for ${\F}=bc$, or $t\geq s$ for ${\F}=BC$.
\end{cor}
\begin{remark}
\label{differential operator properties-rmk}
Let $\boldsymbol{\mathfrak{a}}=(a^r)_r\in \prod_{r=0}^{l}{\F}^{s}(\M,V^{\sigma+\tau+r}_{\tau+\sigma})$ with $\F\in\{bc,BC\}$. Then due to \eqref{section 2: gradient} and Proposition \ref{pointwise multiplication properties}, we deduce that
\begin{align*}
[\boldsymbol{\mathfrak{a}}\mapsto \mathcal{A}(\boldsymbol{\mathfrak{a}})]\in \L(\prod_{r=0}^{l}{\F}^{s}(\M,V^{\sigma+\tau+r}_{\tau+\sigma}),\L({\F}^{s+l}(\M,V),{\F}^{s}(\M,V))).
\end{align*}
\end{remark}
\bigskip

\section{\bf Analytic Semigroups and Continuous Maximal Regularity on Uniformly Regular Riemannian Manifolds without Boundary}

Let $X$ be a Banach space. Given some $\mathcal{E} \geq 1$ and $\vartheta\in [0,\pi)$, a linear differential operator of order $l$, 
\begin{center}
$\mathcal{A}:=\mathcal{A}(x,\partial):=\sum\limits_{|\alpha|\leq{l}}a_{\alpha}(x)\partial^{\alpha}$, 
\end{center}
defined on an open subset $U\subset{\H}^m$ with $a_{\alpha}: U \rightarrow \L(X)$ is said to be $(\mathcal{E},\vartheta;l)$-\emph{elliptic}, if its principal symbol
\begin{center}
$\hat{\sigma}\mathcal{A}^{\pi}(x,\xi):U\times\R^m \rightarrow \L(X)$,\hspace{1em} $(x,\xi)\mapsto \sum\limits_{|\alpha|=l}a_{\alpha}(x)(-i\xi)^{\alpha}$
\end{center}
satisfies
\begin{align}
\label{section 6:unif elpt 1}
\Sigma_{\vartheta}:=\{z\in\mathbb{C}:|{\rm arg}\,z|\leq \vartheta \}\cup\{0\}\subset\rho(-\hat{\sigma}\mathcal{A}^{\pi}(x,\xi))
\end{align}
and  
\begin{align}
\label{section 6:unif elpt 2}
(1+|\lambda|)\|R(\lambda,-\hat{\sigma}\mathcal{A}^{\pi}(x,\xi))\|_{\L(X)}\leq \mathcal{E}, \hspace*{1em}\lambda\in \Sigma_{\vartheta},
\end{align}
and all $(x,\xi)\in U\times\R^m$ with $|\xi|=1$. The constant $\mathcal{E}$ is called the \emph{ellipticity constant} of $\mathcal{A}$. $\hat{\sigma}\mathcal{A}^{\pi}(x,\xi)$ is considered as an element of $\L(X)$. Here $i$ is the complex identity, if necessary, we consider the complexification of $X$. In particular, $\mathcal{A}$ is called \emph{normally elliptic of order $l$} if it is $(\mathcal{E},\frac{\pi}{2};l)$-elliptic for some constant $\mathcal{E}\geq 1$. We readily check that a normally elliptic operator must be of even order. This concept was introduced by H. Amann in \cite{Ama01}. 
\smallskip\\
If $X$ is of finite dimension, then $\mathcal{A}$ is $(\mathcal{E},\vartheta;l)$-elliptic on $U$ iff there exist some $0<r(\mathcal{E},\vartheta)<R(\mathcal{E},\vartheta)$ such that the spectrum of $\hat{\sigma}\mathcal{A}(x,\xi)$ is contained in 
\begin{align*}
\{z\in\mathbb{C}:r<|z|<R\}\cap \mathring{\Sigma}_{\pi-\vartheta}
\end{align*}
for all $(x,\xi)\in U\times\R^m$ with $|\xi|=1$. In particular, in the case that $X=\K$, $\mathcal{A}$ is normally elliptic of order $l$ on $U$ if there exist $0<r(\mathcal{E},\vartheta)<R(\mathcal{E},\vartheta)$ such that
\begin{align*}
{R}\geq Re(\hat{\sigma}\mathcal{A}(x,\xi))\geq{r},\hspace*{1em}(x,\xi)\in{U}\times{\R}^m,\hspace*{.5em}\text{with}\hspace*{.5em}|\xi|=1.
\end{align*}
It is worthwhile mentioning that the concept of normal ellipticity is usually referred to as uniformly strong ellipticity in the scalar-valued case.
\smallskip\\
Let $\ev:\Gamma(\M,V^{\sigma+\tau+l}_{\tau+\sigma}\times V^0_l) \rightarrow \Hom(V)$ be the complete contraction.
A linear operator $\mathcal{A}:=\mathcal{A}(\boldsymbol{\mathfrak{a}}): C^{\infty}({\M}, V) \rightarrow \Gamma({\M}, V)$ of order $l$ is said to be $(\mathcal{E},\vartheta;l)$-\emph{elliptic} if there exists a $\mathcal{E}\geq 1$ such that its principal symbol 
\begin{align*}
\hat{\sigma}\mathcal{A}^{\pi}(p,\xi(p)):=\ev(a^l,(-i\xi)^{\otimes l})(p)\in \L(T_p\M^{\otimes\sigma}\otimes T_p^{\ast}\M^{\otimes\tau})
\end{align*}
satisfies that $S:=\Sigma_{\vartheta}\subset\rho(-\hat{\sigma}\mathcal{A}^{\pi}(p,\xi(p)))$ and
\begin{align*}
(1+|\lambda|)\|R(\lambda,-\hat{\sigma}\mathcal{A}^{\pi}(p,\xi(p))) 
\|_{\L(T_p\M^{\otimes\sigma}\otimes T_p^{\ast}\M^{\otimes\tau})} \leq \mathcal{E},\hspace*{1em}\lambda\in S,
\end{align*}
for all $(p,\xi)\in \M\times\Gamma(\M, T^\ast M)$ with $|\xi|_{g^{\ast}}=1_{\M}$.
This definition is a natural extension of its Euclidean version. In fact, a linear differential $\mathcal{A}$ of order $l$ is $(\mathcal{E},\vartheta;l)$-\emph{elliptic}, iff all its local realizations 
\begin{align*}
\mathcal{A}_{\kappa}(x,\partial)=\sum\limits_{|\alpha|\leq{l}}a^{\kappa}_{\alpha}(x)\partial^{\alpha}
\end{align*}
are $(\mathcal{E}^{\prime},\vartheta;l)$-elliptic on $\Qk$ with uniform constants $\mathcal{E}^{\prime}$, $\vartheta$ in condition \eqref{section 6:unif elpt 1} and \eqref{section 6:unif elpt 2}. In particular, $\mathcal{A}$ is called \emph{normally elliptic of order $l$}, if $\mathcal{A}$ is $(\mathcal{E},\frac{\pi}{2};l)$-elliptic for some constant $\mathcal{E}\geq 1$. Analogously, the constant $\mathcal{E}$ is called the \emph{ellipticity constant} of $\mathcal{A}$ defined on $\M$. 

For the proof of the main theorem, we first quote a result of H. Amann.
\begin{prop}
\cite[Theorem~4.2, Remark~4.6]{Ama01}
\label{S3: RE-Rm}
Let $s> 0$, $s\notin\N$ and $X$ be a Banach space. Suppose that $\mathcal{A}=\sum_{|\alpha|\leq l}a_{\alpha}\partial^{\alpha}$ is a $(\mathcal{E},\vartheta;l)$-elliptic operator on $\R^m$ with 
\begin{center}
$a_{\alpha}\in bc^s(\R^m,\L(X))$, and $\|a_{\alpha}\|_{s,\infty}\leq \mathcal{K}$ for all $|\alpha|\leq l$.
\end{center}
Then there exist constants $\omega(\mathcal{E},\mathcal{K})$, $\mathcal{N}(\mathcal{E},\mathcal{K})$ such that $S:=\omega+\Sigma_{\vartheta}\subset \rho(-\mathcal{A})$ and
\begin{align*}
|\lambda|^{1-j}\|R(\lambda,-\mathcal{A})\|_{\L(\F^s(\R^m,X),\F^{s+lj}(\R^m,X))}\leq \mathcal{N},\hspace{1em}\text{for } \lambda\in S,\hspace*{.5em} j=0,1,
\end{align*}
with $\F\in\{bc,BC\}$. Here $\mathcal{E}$ is the ellipticity constant of $\mathcal{A}$.
\end{prop}

\begin{theorem}
\label{generator of analytic semigroup-tensor}
Let $s> 0$ and $s\notin\N$. Suppose that ${\M}$ is a uniformly regular Riemannian manifold without boundary, and $\mathcal{A}$ is a $(\mathcal{E},\vartheta;2l)$-elliptic operator with $\boldsymbol{\mathfrak{a}}=(a^r)_r\in \prod_{r=0}^{2l} bc^{s}(\M,V^{\sigma+\tau+r}_{\tau+\sigma})$. 
Then there exist constant $\omega(\mathcal{E},\mathcal{K})$, $\mathcal{N}(\mathcal{E},\mathcal{K})$ such that $S:=\omega+\Sigma_{\vartheta}\subset \rho(-\mathcal{A})$ and
\begin{align*}
|\lambda|^{1-j}\|R(\lambda,-\mathcal{A})\|_{\L(\F^s(\M,V),\F^{s+2lj}(\M,V))}\leq \mathcal{N},\hspace{1em}\text{for } \lambda\in S,\hspace*{.5em} j=0,1
\end{align*}
with $\F\in\{bc,BC\}$. Here $\mathcal{E}$ is the ellipticity constant of $\mathcal{A}$ and $\mathcal{K}:=\max\limits_r \|a^r\|^{\M}_{s,\infty}$.
\end{theorem}
\begin{proof}
To economize notations, we set 
\begin{align*}
E_0=\F^{s},\hspace{.5em} E_{\theta}=\F^{s+2l-1},\hspace{.5em}\text{ and } E_1=\F^{s+2l}.
\end{align*}
In virtue of Proposition~\ref{differential operator coef} and the discussions at the beginning of this section, we know that there exist constants $\mathcal{E}^{\prime}$ and $\mathcal{K}^{\prime}$ such that all localizations $\mathcal{A}_{\kappa}$'s are $(\mathcal{E}^{\prime},\vartheta;2l)$-elliptic and their coefficients satisfy
\begin{equation*}
(a^{\kappa}_{\alpha})_{\kappa}\in l_{\infty,\uf}({bc^{s}}({\Qk},\L(E)))
\quad\text{with}\quad
 \max\limits_{|\alpha|\leq 2l}\sup\limits_{\kappa}\|a^{\kappa}_{\alpha}\|_{s,\infty}\leq \mathcal{K}^{\prime}. 
\end{equation*}
Without loss of generality, we may assume that $\mathcal{E}=\mathcal{E}^{\prime}$ and $\mathcal{K}=\mathcal{K}^{\prime}$.

(i) For simplicity, we first assume that $s\in (0,1)$. It can be easily seen through step (ii) of \cite[Theorem~4.1]{Ama01} that this assumption will not harm our proof. Define $h:{\R}^m\rightarrow\Q$: $x\mapsto\zeta(x)x$. It is easy to see that $h\in{BC^{\infty}({\R}^m,{\Q})}$. Let
\begin{equation*}
\bar{\mathcal{A}}_{\kappa}(x,\partial):=\sum\limits_{|\alpha|\leq{2l}}\bar{a}_{\alpha}^{\kappa}(x)\partial^{\alpha}:=\sum\limits_{|\alpha|\leq{2l}}(a_{\alpha}^{\kappa}\circ h)(x)\partial^{\alpha}.
\end{equation*}
It is not hard to check that $(\bar{a}_{\alpha}^{\kappa})_{\kappa}\in l_{\infty,\uf}({\boldsymbol{bc}^{s}}(\R^m,\L(E)))$ and $\bar{\mathcal{A}}_{\kappa}$ is $(\mathcal{E},\vartheta;2l)$- elliptic on ${\R}^m$. 

(ii) For any $\lambda\in\mathbb{C}$ and $u\in{E_1}$, consider
\begin{align*}
\hspace*{1.2em} \Rck(\lambda+\mathcal{A})u-(\lambda+\bar{\mathcal{A}}_{\kappa})\Rck u
&={\kfk}(\pi_{\kappa}(\lambda+\mathcal{A})u)-(\lambda+\bar{\mathcal{A}}_{\kappa}){\kfk}(\pi_{\kappa}u)\\
&={\kfk}\pi_{\kappa}(\lambda+\bar{\mathcal{A}}_{\kappa}){\kfk}u-(\lambda+\bar{\mathcal{A}}_{\kappa}){\kfk}(\pi_{\kappa}u) \\
&={\kfk}\pi_{\kappa}\bar{\mathcal{A}}_{\kappa}{\kfk}u-\bar{\mathcal{A}}_{\kappa}{\kfk}(\pi_{\kappa}u)\\
&=-\sum\limits_{|\alpha|\leq{2l}}\sum\limits_{0<\beta\leq\alpha}\binom{\alpha}{\beta}\bar{a}_{\alpha}^{\kappa}\partial^{\alpha-\beta}(\zeta{\kfk}u)\partial^{\beta}({\kfk}\pi_{\kappa}).
\end{align*}
In the rest of the proof, we always conventionally put
\begin{align*}
b=``\infty" \text{ for }\F=BC,\text{ or }b=``\infty,\uf" \text{ for }\F=bc.
\end{align*}
Using the proof of \cite[Theorem~12.1]{Ama12} and interpolation theory, one 
readily checks that for $t\geq 0$
\begin{align}
\label{S3: pull onto Euc by zeta}
[u\mapsto (\zeta {\kfk}u)_{\kappa}]\in\L(\F^t(\M,V),l_{b}(\boldsymbol{\F}^t)).
\end{align}
By Proposition~\ref{pointwise multiplication properties}, the following map is bilinear and continuous for $\F\in\{bc,BC\}$
\begin{align}
\label{S3: pointwise-mul-Rm}
({\F^t(\R^m,\L(E))},{\F^t(\R^m,E)})\rightarrow{\F^t(\R^m,E)}:(F,v)(x)\mapsto F(x)(v(x))
\end{align}
for $t\notin\N_0$. 
In view of \eqref{S3: pull onto Euc by zeta} and \eqref{S3: pointwise-mul-Rm}, we conclude that
\begin{equation*}
\Rck((\lambda+\mathcal{A})\cdot)-(\lambda+\bar{\mathcal{A}}_{\kappa})\Rck(\cdot):=B_{\kappa}(\cdot)\in\mathcal{L}(E_{\theta}(\M,V),E_{0}(\R^m,E)).
\end{equation*}  
Analogously, we can verify that
\begin{equation*}
(\lambda+\mathcal{A})\Rek(\cdot)-\Rek((\lambda+\bar{\mathcal{A}}_{\kappa})\cdot):=C_{\kappa}(\cdot)\in\mathcal{L}(E_{\theta}(\R^m,E),E_0(\M,V)).
\end{equation*}
\smallskip 
(iii) Set $\bar{\mathcal{A}}:l_{b}(\boldsymbol{E}_{1})\rightarrow \prod_{\kappa}\boldsymbol{\F}^s_{\kappa}$: $(u_{\kappa})_{\kappa}\mapsto(\bar{\mathcal{A}}_{\kappa}u_{\kappa})_{\kappa}$. By Proposition \ref{pointwise multiplication properties} and \eqref{section 2: Holder/bundle}, one attains
\begin{align*}
\bar{\mathcal{A}}\in\mathcal{L}(l_{b}(\boldsymbol{E}_{1}),l_{b}(\boldsymbol{E}_{0})).
\end{align*}
Put $\boldsymbol{u}=(u_{\kappa})_{\kappa}$ and $\boldsymbol{a}_{\alpha}=(\bar{a}_{\alpha}^{\kappa})_{\kappa}$. Define $\boldsymbol{\bar{\mathcal{A}}}:E_1(\R^m,l_{\infty}(\boldsymbol{E}))\rightarrow E_0(\R^m,l_{\infty}(\boldsymbol{E}))$ by
\begin{align*}
\boldsymbol{\bar{\mathcal{A}}}:=\sum\limits_{|\alpha|\leq 2l}\f^{-1}(\boldsymbol{a}_{\alpha})\partial^{\alpha}.
\end{align*}
Then by \cite[formula~(11.29)]{Ama12}, we obtain 
\begin{align}
\label{S3: F-A}
\f^{-1}(\bar{\mathcal{A}}\boldsymbol{u})=\boldsymbol{\bar{\mathcal{A}}} \f^{-1}(\boldsymbol{u})=\sum\limits_{|\alpha|\leq 2l}\f^{-1}(\boldsymbol{a}_{\alpha})\partial^{\alpha}\f^{-1}(\boldsymbol{u}).
\end{align}
Theorem \ref{retraction of BCk} implies that $\f^{-1}(\boldsymbol{a}_{\alpha})\in bc^s(\R^m,l_{\infty}(\L(E)))\hookrightarrow bc^s(\R^m,\L(l_{\infty}(\boldsymbol{E})))$. By the $(\mathcal{E},\vartheta;2l)$-ellipticity of $\bar{\mathcal{A}}_{\kappa}$ and the uniform ellipticity constant $\mathcal{E}$ in \eqref{section 6:unif elpt 1}, we conclude that $\boldsymbol{\bar{\mathcal{A}}}$ is $(\mathcal{E},\vartheta;2l)$-elliptic, i.e., for all $(x,\xi)\in\R^m\times\R^m$ with $|\xi|=1$
\begin{align*}
(1+|\lambda|)\|R(\lambda,-\hat{\sigma}\boldsymbol{\bar{\mathcal{A}}}^{\pi}(x,\xi))\|_{\L(l_{\infty}(\boldsymbol{E}))}\leq \mathcal{E}, \hspace{1em}\lambda\in\Sigma_{\vartheta}.
\end{align*}
Set $\hat{E}_j:=E_j(\R^m,l_{\infty}(\boldsymbol{E}))$ for $j=0,1$. Proposition \ref{S3: RE-Rm} now yields the existence of constants $\omega_0$ and $M_0$ such that $S_0:=\omega_0+\Sigma_{\vartheta}\subset \rho(-\boldsymbol{\bar{\mathcal{A}}})$ and 
\begin{align*}
|\lambda|^{1-j}\|R(\lambda,-\boldsymbol{\bar{\mathcal{A}}}\|_{\L(\hat{E}_0,\hat{E}_j)}\leq M_0, \hspace*{1em}\lambda\in S_0,\hspace*{.5em}j=0,1.
\end{align*}
Pick $\boldsymbol{u}\in l_{b}(\boldsymbol{E}_1)$. By \eqref{section 2:Lis bc^s} and \eqref{S3: F-A} we have for every $\lambda\in S_0$
\begin{align*}
\f(\lambda+\boldsymbol{\bar{\mathcal{A}}})^{-1}\f^{-1}(\lambda+\bar{\mathcal{A}})\boldsymbol{u}=\f(\lambda+\boldsymbol{\bar{\mathcal{A}}})^{-1}(\lambda+\boldsymbol{\bar{\mathcal{A}}})\f^{-1}\boldsymbol{u}=\boldsymbol{u}.
\end{align*}
Similarly, by \cite[formula~(11.29)]{Ama12}
\begin{align*}
(\lambda+\bar{\mathcal{A}})\f(\lambda+\boldsymbol{\bar{\mathcal{A}}})^{-1} \f^{-1}\boldsymbol{u}=\f (\lambda+\boldsymbol{\bar{\mathcal{A}}})(\lambda+\boldsymbol{\bar{\mathcal{A}}})^{-1}\f^{-1}\boldsymbol{u}=\boldsymbol{u}.
\end{align*}
Thus $S_0\subset \rho(-\bar{\mathcal{A}})$. Furthermore, $(\lambda+\bar{\mathcal{A}})^{-1}=\f(\lambda+\boldsymbol{\bar{\mathcal{A}}})^{-1}\f^{-1}$ and 
\begin{align}
\label{section 6: type of bA-LH}
|\lambda|^{1-j}\|R(\lambda,-\bar{\mathcal{A}})\|_{\L(l_{b}(\boldsymbol{E}_0),l_{b}(\boldsymbol{E}_j))} \leq {K_0}, \hspace*{1em}\lambda\in S_0,\hspace*{.5em}j=0,1
\end{align}
for some $K_0 \geq 1$. 

(iv) Define $B:E_{\theta}(\M,V)\rightarrow \prod_{\kappa}\boldsymbol{\F}^s_{\kappa}$: $u\mapsto(B_{\kappa}u)_{\kappa}$. By \eqref{S3: pull onto Euc by zeta}, we have
\begin{center}
$B\in\mathcal{L}(E_{\theta}(\M,V),l_{b}(\boldsymbol{E}_{0}))$.
\end{center}
By Theorem \ref{retraction of BCk}, we have that $B\Re\in\mathcal{L}(l_{b}(\boldsymbol{E}_{\theta}),l_{b}(\boldsymbol{E}_{0}))$. It follows from \eqref{section 2:Lis bc^s}, \eqref{section 2:Lis BC^s} and \cite[Proposition~I.2.3.2]{Ama95} that 
\begin{align*}
l_{b}(\boldsymbol{E}_{\theta})\doteq (l_{b}(\boldsymbol{E}_{0}),l_{b}(\boldsymbol{E}_{1}))_{\theta},
\end{align*}
where either $(\cdot,\cdot)_{\theta}=(\cdot,\cdot)^0_{\theta,\infty}$ for $\F=bc$, or $(\cdot,\cdot)_{\theta}=(\cdot,\cdot)_{\theta,\infty}$ for $\F=BC$, and $\theta=1-1/(2l)$. 
Thus $B\Re$ is a lower order perturbation. For each $\boldsymbol{u}\in l_{b}(\boldsymbol{E}_{0})$ and $\lambda\in S_0$, one computes
\begin{align*}
&\hspace*{1.2em} \|B\mathcal{R}(\lambda+\bar{\mathcal{A}})^{-1}\boldsymbol{u}\|_{l_{b}(\boldsymbol{E}_{0})}\\
 &\leq M \varepsilon\|(\lambda+\bar{\mathcal{A}})^{-1}\boldsymbol{u}\|_{l_{b}(\boldsymbol{E}_{1})}+M C(\varepsilon)\|(\lambda+\bar{\mathcal{A}})^{-1}\boldsymbol{u}\|_{l_{b}(\boldsymbol{E}_{0})}\\
&\leq M K_0  \varepsilon \|\boldsymbol{u}\|_{l_{b}(\boldsymbol{E}_{0})} +M C(\varepsilon) \frac{K_0}{|\lambda|}\|\boldsymbol{u}\|_{l_{b}(\boldsymbol{E}_{0})}.
\end{align*}
The penultimate line is a consequence of interpolation theory \cite[formula~(I.2.2.2)]{Ama95}.
The last inequality follows from \eqref{section 6: type of bA-LH}. By choosing $\varepsilon$ small enough and a sufficiently large $\omega_1\geq \omega_0$, we can conclude that $\|B\mathcal{R}(\lambda+\bar{\mathcal{A}})^{-1}\|\leq \frac{1}{2}$. By a Neumann series argument, we infer that $S_1:=\omega_1+\Sigma_{\vartheta}\subset \rho(-\bar{\mathcal{A}}-B\Re)$ and
\begin{align*}
|\lambda|^{1-j}\|R(\lambda,-\bar{\mathcal{A}}-B\Re)\|_{\L(l_{b}(\boldsymbol{E}_0),l_{b}(\boldsymbol{E}_j))}\leq M_1, \hspace*{1em}\lambda\in S_1,\hspace*{.5em}j=0,1,
\end{align*}
for some $M_1$. For any $\lambda\in{S_1}$, $(\lambda+\bar{\mathcal{A}}+B\Re)^{-1}$ exists. Moreover,
\begin{align*} 
\Re(\lambda+\bar{\mathcal{A}}+B\Re)^{-1}\Rc(\lambda+\mathcal{A})
=\Re(\lambda+\bar{\mathcal{A}}+B\Re)^{-1}(\lambda+\bar{\mathcal{A}}+B\Re)\Rc={\id}_{E_1(\M,V)}.
\end{align*}
Therefore, $\lambda+\mathcal{A}$ is injective.

(v) Define $C:l_{b}(\boldsymbol{E}_{\theta})\rightarrow{E_0(\M,V)}$: $(u_{\kappa})_{\kappa}\mapsto{\sum\limits_{\kappa}C_{\kappa}(u_{\kappa})}$. Then analogously we can verify that $C\in\mathcal{L}(l_{b}(\boldsymbol{E}_{\theta}),E_0(\M,V))$ and in turn $\Rc{C}\in\mathcal{L}(l_{b}(\boldsymbol{E}_{\theta}),l_{b}(\boldsymbol{E}_{0}))$. 
\smallskip\\
For any $\boldsymbol{u} \in l_{b}(\boldsymbol{E}_{1})$, we have
\begin{align*}
(\lambda+\mathcal{A})\Re \boldsymbol{u}=\Re(\lambda+\bar{\mathcal{A}})\boldsymbol{u}+\Re\Rc{C}\boldsymbol{u}=\Re(\lambda+\bar{\mathcal{A}}+\Rc{C})\boldsymbol{u}.
\end{align*}
An analogous argument to (iv) reveals that there exist $\omega\geq \omega_1$ and $\mathcal{N}>M_1$ such that $S:=\omega+\Sigma_{\vartheta} \subset\rho(-\bar{\mathcal{A}}-\Rc{C})$ and
\begin{align*}
|\lambda|^{1-j}\|R(\lambda,-\bar{\mathcal{A}}-\Rc C)\|_{\L(l_{b}(\boldsymbol{E}_0),l_{b}(\boldsymbol{E}_j))}\leq \mathcal{N}, \hspace*{1em}\lambda\in S,\hspace*{.5em}j=0,1.
\end{align*}
Moreover, these constants depend only on $\mathcal{E}$ and $\mathcal{K}$. For any $\lambda\in S$, $(\lambda+\bar{\mathcal{A}}+\Rc C)^{-1}$ exists and $(\lambda+\mathcal{A})\Re(\lambda+\bar{\mathcal{A}}+\Rc C)^{-1}\Rc=\id_{E_0(\M,V)}$. Thus $(\lambda+\mathcal{A})$ is surjective. Now we conclude that $S \subset\rho(-\mathcal{A})$, and for every $\lambda\in S$, $u\in E_0(\M,V)$
\begin{align*}
|\lambda|^{1-j}\|R(\lambda,-\mathcal{A})u\|_{E_j}^{\M} &=|\lambda|^{1-j}\|\Re R(\lambda,-\bar{\mathcal{A}}-\Rc C) \Rc u\|_{E_j}^{\M}\\
&\leq C_0|\lambda|^{1-j}\|R(\lambda,-\bar{\mathcal{A}}-\Rc C) \Rc u\|_{l_b(\boldsymbol{E}_j)} \leq C^2_0 \mathcal{N}\|u\|_{E_0}^{\M}.
\end{align*}
\end{proof}
\begin{remark}
With the necessary modification, the above proof also works for Banach-valued functions defined on a uniformly regular manifold without boundary, provided the counterparts of these spaces and the elliptic conditions are properly defined.
\end{remark}

Recall an operator ${A}$ is said to belong to the class $\mathcal{H}(E_1,E_0)$ for some densely embedded Banach couple $E_1\overset{d}{\hookrightarrow}E_0$, if $-{A}$ generates a strongly continuous analytic semigroup on $E_0$ with $dom(-{A})=E_1$. By the well-known semigroup theory, Theorem \ref{generator of analytic semigroup-tensor} immediately implies
\begin{theorem}
\label{S3:main theorem}
Suppose $\mathcal{A}$ is normally elliptic of order $2l$ and the coefficients $\boldsymbol{\mathfrak{a}}=(a^r)_r$ satisfy the conditions in Theorem \ref{generator of analytic semigroup-tensor}. Then
\begin{align*}
\mathcal{A}\in\mathcal{H}(bc^{s+2l}(\M,V),bc^{s}(\M,V)).
\end{align*}
\end{theorem}

For some fixed interval $I=[0,T]$, $\gamma\in(0,1]$ and some Banach space $X$, we define
\begin{align*}
&BU\!C_{1-\gamma}(I,X):=\{u\in{C(\dot{I},X)};[t\mapsto{t^{1-\gamma}}u]\in{BU\!C(\dot{I},X)},\lim\limits_{t\to{0^+}}{t^{1-\gamma}}\|u\|=0\},\\
& \|u\|_{C_{1-\gamma}}:=\sup_{t\in{\dot{I}}}{t^{1-\gamma}}\|u(t)\|_{X},
\end{align*}
and
\begin{center}
$BU\!C_{1-\gamma}^1(I,X):=\{u\in{C^1(\dot{I},X)}: u,\dot{u}\in{BU\!C_{1-\gamma}(I,X)}\}$.
\end{center}
In particular, we put
\begin{center}
$BU\!C_0(I,X):=BU\!C(I,X)$\hspace*{1em} and \hspace*{1em} $BU\!C^1_0(I,X):=BU\!C^1(I,X)$.
\end{center}
In addition, if $I=[0,T)$ is a half open interval, then
\begin{align*}
&C_{1-\gamma}(I,X):=\{v\in{C(\dot{I},X)}:v\in{BU\!C_{1-\gamma}([0,t],X)},\hspace{.5em} t<T\},\\
&C^1_{1-\gamma}(I,X):=\{v\in{C^1(\dot{I},X)}:v,\dot{v}\in{C_{1-\gamma}(I,X)}\}.
\end{align*}
We equip these two spaces with the natural Fr\'echet topology induced by the topology of $BU\!C_{1-\gamma}([0,t],X)$ and $BU\!C_{1-\gamma}^1([0,t],X)$, respectively. 
\smallskip\\
Assume that $E_1\overset{d}{\hookrightarrow}E_0$ is a densely embedded Banach couple. Define
\begin{align}
\label{S3: ez&ef}
{\ez}(I):=BU\!C_{1-\gamma}(I,E_0), \hspace{1em} {\ef}(I):=BU\!C_{1-\gamma}(I,E_1)\cap{BU\!C_{1-\gamma}^1(I,E_0)}.
\end{align}
For ${A}\in\mathcal{H}(E_1,E_0)$, we say $({\ez}(I),{\ef}(I))$ is a pair of maximal regularity of ${A}$, if
\begin{center}
$(\frac{d}{dt}+{A},\gamma_{0})\in{\Lis}({\ef}(I),{\ez}(I)\times{E_{\gamma}})$, 
\end{center}
where $\gamma_{0}$ is the evaluation map at $0$, i.e., $\gamma_{0}(u)=u(0)$, and $E_{\gamma}:=(E_0,E_1)_{\gamma,\infty}^0$. Symbolically, we denote it by
\begin{align*}
{A}\in \mathcal{M}_{\gamma}(E_1,E_0).
\end{align*}
Let $E_2:=(dom({A}^2),\|\cdot\|_{E_2})$, where $\|\cdot\|_{E_2}:=\|{A}\cdot\|_{E_1}+\|\cdot\|_{E_1}$. Put
\begin{itemize}
\item $E_{1+\theta}:=(E_1,E_2)_{\theta,\infty}^0$,\hspace{1em}$\theta\in (0,1)$,
\item ${A}_{\theta}:=$ the maximal $E_{\theta}$-realization of ${A}$.
\end{itemize}
The next step is to conclude a maximal regularity result for normally elliptic operators. To this end, we quote a famous theorem, which was first proved by G. Da~Prato and P. Grisvard \cite{DaPra79}, and then generalized later by S.~Angenent.
\begin{theorem}{\cite[Theorem~2.14]{Ange90}}
\label{Da Prato, Grisvard}
Suppose $A\in \mathcal{H}(E_1,E_0)$ and let $\gamma\in(0,1]$. Then
\begin{center}
$(\mathbb{E}_{\theta}(I),\mathbb{E}_{1+\theta}(I)):=(BU\!C_{1-\gamma}(I,E_{\theta}),BU\!C_{1-\gamma}(I,E_{1+\theta})\cap{BU\!C_{1-\gamma}^1(I,E_{\theta})})$
\end{center}
is a pair of maximal regularity for ${A}_{\theta}$, that is,
\begin{center}
$(\frac{d}{dt}+{A}_{\theta},\gamma_{0})\in{\Lis}(\mathbb{E}_{1+\theta}(I),\mathbb{E}_{\theta}(I)\times(E_{\theta},E_{1+\theta})_{\gamma,\infty}^0)$.
\end{center}
\end{theorem}
\smallskip\goodbreak
\begin{theorem}
\label{section 6: H-MR}
Let $\gamma\in(0,1]$ and $s\notin\N_0$. Suppose that $\mathcal{A}$ satisfies the conditions in Theorem \ref{S3:main theorem}. Then
\begin{center}
$\mathcal{A}\in\mathcal{M}_{\gamma}(E_1,E_0)$
\end{center}
with $E_0:=bc^{s}(\M,V)$ and $E_1:=bc^{s+2l}(\M,V)$.
\end{theorem}
\begin{proof}
Pick $\max\{0,s-2l\}<\alpha<s$ with $\alpha\notin\N_0$. Let $F_0=bc^{\alpha}(\M,V)$, $F_1=bc^{\alpha+2l}(\M,V)$. Set $\theta=\frac{s-\alpha}{2l}$. By Proposition \ref{interpolation of LH spaces}(c), we have $F_{\theta}=E_0$ and $\mathcal{A}\in\mathcal{H}(F_1,F_0)$.
\smallskip\\
By the preceding theorem, we infer that $\mathcal{A}_{\theta}\in\mathcal{M}_{\gamma}(F_{1+\theta},E_0)$. In particular, $\mathcal{A}_{\theta}\in\mathcal{H}(F_{1+\theta},E_0)$. Note that $dom(\mathcal{A})=dom(\mathcal{A}_{\theta})$.
\smallskip\\
By analytic semigroup theory, there exists a $\lambda>0$ such that
\begin{center}
$\lambda+\mathcal{A}\in{\Lis}(F_{1+\theta},E_0)\cap{\Lis}(E_1,E_0)$.
\end{center}
Consequently, it implies that $F_{1+\theta}\doteq E_1$ This completes the proof.
\end{proof}
\begin{prop}
Suppose that all the local representations $\mathcal{A}_{\kappa}=\sum\limits_{|\alpha|\leq 2l}a^{\kappa}_{\alpha}\partial^{\alpha}$ of $\mathcal{A}$ are normally elliptic (or $(\mathcal{E},\vartheta;2l)$-elliptic with uniform constants $\mathcal{E},\vartheta$, respectively), and their coefficients satisfy 
\begin{center}
$(a^{\kappa}_{\alpha})_{\kappa}\in l_{\infty}({BC^{t}}({\Qk},\L(E)))$, \hspace{.5em}with \hspace{.5em}$\max\limits_{|\alpha|\leq 2l}\sup\limits_{\kappa}\|a^{\kappa}_{\alpha}\|_{s,\infty}\leq \mathcal{K}$  
\end{center}
for some $\mathcal{K}>0$ and $t>s$. Then the assertions in Theorem~\ref{S3:main theorem} and \ref{section 6: H-MR} (or Theorem~\ref{generator of analytic semigroup-tensor}, respectively) hold true.
\end{prop}
\bigskip

\section{\bf Parameter-Dependent Diffeomorphisms on Uniformly Regular Riemannian Manifolds}
The main purpose of this section is to introduce a family of parameter-dependent diffeomorphisms induced by a truncated translation technique. As will be shown later, this technique combined with the implicit function theorem serves as a crucial tool to study regularity of solutions to parabolic equations on uniformly regular Riemannian manifolds. The idea of a family of truncated translations was introduced in \cite{EscPruSim031} by J.~Escher, J.~Pr\"uss and G.~Simonett to establish regularity for solutions to parabolic and elliptic equations in Euclidean spaces. The major obstruction of bringing in the localized translations for uniformly regular Riemannian manifolds lies in how to introduce parameters so that the transformed functions and differential operators depend ideally on the parameters as long as they are smooth enough around the ``center" of the localized translations. Thanks to the discussions in the previous section, we are able to set up these properties based on their counterparts in \cite{EscPruSim031}.

Suppose that $(\M,g)$ is a uniformly regular Riemannian manifold, with a uniformly regular atlas $\mathfrak{A}$. 
Given any point $p \in\mathring{\M}$, there is a local chart $(\Ukp,\varphi_{\kappa_{\sf p}})\in\mathfrak{A}$ containing $p$. 
Let $\x:=\varphi_{\kappa_{\sf p}}(p)$ and 
$d:={\rm{dist}}(\x,\partial \mathbb{B}^m_{\kappa_{\sf p}})$.
We define a new local patch $(\Uki,\varphi_{\iota})$ around $p$ 
by means of
\begin{equation*}
\Uki:=\psi_{\kappa_{\sf p}}(\mathbb{B}(\x,d)),
\quad
\varphi_{\iota}(q):=\frac{\varphi_{\kappa_{\sf p}}(q)-\x}{d},\quad q\in\Ukp.
\end{equation*}
Then  $\varphi_{\iota}(p)=0\in\R^m$, $\varphi_{\iota}(\Uki)=\Q$, and
the transition maps between $(\Uki,\varphi_{\iota}) $ and $(\Ukp,\varphi_{\kappa_{\sf p}})$ are given by
\begin{align*}
\varphi_{\iota}\circ\psi_{\kappa_{\sf p}}(y)=\frac{y-\x}{d},\quad 
\varphi_{\kappa_{\sf p}}\circ\psi_{\iota}(x)=d\,x+\x, \hspace*{1em} x\in\Q,\; 
y\in \mathbb{B}(\x,d).
\end{align*}
Choose $\varepsilon_0>0$ small such that $5\varepsilon_0<1$ and set
\begin{align*}
B_i:=\mathbb{B}^{m}(0,i\varepsilon_{0}),\hspace*{1em}\text{ for }i=1,2,...,5.
\end{align*}

As part of the preparations for introducing a family of parameter-dependent diffeomorphisms on ${\M}$, we pick two cut-off functions on ${\Qi}$:
\begin{itemize}
\item  $\chi\in{\mathcal{D}(B_{2},[0,1])}$ such that $\chi|_{\bar{B_{1}}}\equiv{1}$. 
\item  ${\Xo}\in\mathcal{D}(B_5,[0,1])$ such that ${\Xo}|_{\bar{B_4}}\equiv{1}$. We write ${\Xt}={\kb}{\Xo}$.
\end{itemize}
\smallskip
We define a rescaled translation on ${\Q}$ for any ${\mu}\in{\B}\subset{\R}^m$ with $r>0$ sufficiently small:
\begin{center}
$\theta_{\mu}(x):=x+\chi{(x)}\mu$,\hspace{1em} $x\in{\Q}$.
\end{center}
This localization technique in Euclidean spaces was first introduced in \cite{EscPruSim031} by J.~Escher, J.~Pr\"uss and G.~Simonett. $\theta_{\mu}$ induces a transformation $\Theta_{\mu}$ on ${\M}$ by:
\begin{align*}
\Theta_{\mu}(q)=
\begin{cases}
\psi_{\iota}(\theta_{\mu}(\varphi_{\iota}(q))) \hspace{1em}&q\in{\Uki},\\
q &q\notin{\Uki}.
\end{cases}
\end{align*}
In particular, we have $\Theta_{\mu}\in{\Dfi}({\M})$ for $\mu\in{\B}$ with sufficiently small $r>0$.
\smallskip\\
We may find an explicit global expression for ${\ttm}$, i.e., the pull-back of $\Theta_{\mu}$. Given $u\in \Gamma({\M},V)$,
\begin{equation*}
{\ttm}u={\kb}{\tm}{\kf}({\Xt}u)+(1-{\Xt})u.
\end{equation*}
Likewise, ${\Theta}^{\mu}_{\ast}$ can be expressed by
\begin{equation*}
{\Theta}^{\mu}_{\ast}u={\kb}{\theta^{\mu}_{\ast}}{\kf}({\Xt}u)+(1-{\Xt})u.
\end{equation*}
\smallskip 
Let $I=[0,T]$, $T>0$. Assuming that $t_{0}\in\mathring{I}$ is a fixed point, we choose $\varepsilon_{0}$ so small that $\mathbb{B}(t_{0},3\varepsilon_{0})\subset\mathring{I}$. Next pick another auxiliary function 
\begin{center}
$\xi\in\mathcal{D}(\mathbb{B}(t_{0},2\varepsilon_{0}),[0,1])$\hspace{1em} with\hspace{.5em} $\xi|_{\mathbb{B}(t_{0},\varepsilon_{0})}\equiv{1}$. 
\end{center}
The above construction now engenders a parameter-dependent transformation in terms of the time variable:
\begin{center}
$\varrho_{\lambda}(t):=t+\xi(t)\lambda$,\hspace{.5em} for any $t\in{I}$ and $\lambda\in{\R}$.
\end{center}
\smallskip
Now we are in a situation to define a family of parameter-dependent transformations on $I\times{\M}$. Given a function $u:I\times{\M}\rightarrow V$, we set
\begin{center}
${{u}_{\lambda,\mu}}(t,\cdot):={\ttl}u(t,\cdot):={\tu}(t){\rh}u(t,\cdot)$,
\end{center}
where ${\tu}(t):={\Theta}^{\ast}_{\xi(t)\mu}$ and $(\lambda,\mu)\in{\B}$.
\smallskip\\
It is important to note that ${{u}_{\lambda,\mu}}(0,\cdot)=u(0,\cdot)$ for any $(\lambda,\mu)\in{\B}$ and any function $u$. Here and in the following, I will not distinguish between ${\B}$, $\mathbb{B}^{m}(0,r)$ and $\mathbb{B}^{m+1}(0,r)$. As long as the dimension of the ball is clear from the context, we always simply write them as ${\B}$.

The importance of the family of parameter-dependent diffeomorphisms $\{{\ttl}:(\lambda,\mu)\in\B\}$ lies in the following results. Their proofs and the additional properties of this technique can be found in \cite{ShaoPre}. Let $\omega$ be the symbol for real analyticity. It is understood that under the condition that $k=\omega$ in the following theorems, $\M$ is assumed to be a $C^{\omega}$-uniformly regular Riemannian manifold.
\begin{theorem}
\label{S3: inverse}
Let $k\in{\N}\cup\{\infty,\omega\}$. Suppose that $u\in{BC(I\times{\M},V)}$. Then  $u\in{C^{k}(\mathring{I}\times\mathring{\M},V)}$ iff for any $(t_0,p)\in\mathring{I}\times\mathring{\M}$, there exists $r=r(t_0,p)>0$ and a corresponding family of parameter-dependent diffeomorphisms ${\ttl}$ such that 
\begin{center}
$[(\lambda,\mu)\mapsto{\ttl}u]\in{C^{k}({\B},BC(I\times{\M},V))}$.
\end{center}
\end{theorem}
\goodbreak
Henceforth, assume that $\F\in\{bc,BC\}$. Let $E_0:={\F}^{s}(M,V)$ and $E_1:={\F}^{s+l}(M,V)$ with $l\in\N$. Define ${\ef}(I)$ as in \eqref{S3: ez&ef} by fixing $\theta=1$. 
\begin{prop}
\label{S3: time derivatives}
Suppose that $u\in{\ef}(I)$. Then $u_{\lambda,\mu} \in{\ef}(I)$, and 
\begin{align*}
\partial_t[u_{\lambda,\mu}]=(1+\xi^{\prime}\lambda){\ttl}u_t+B_{\lambda,\mu}(u_{\lambda,\mu}), \hspace*{1em}(\lambda,\mu)\in\B,
\end{align*}
where $[(\lambda,\mu)\mapsto{B}_{\lambda,\mu}]\in{C}^{\omega}({\B},C(I,\mathcal{L}({\F}^{s+l}(M,V),{\F}^{s}(M,V))))$. Furthermore, $B_{\lambda,0}=0$.
\end{prop}
\goodbreak
\begin{prop}
\label{S3: regularity of differential operators}
Let $n\geq 0$, $k\in{\N}_{0}\cup\{\infty,\omega\}$, and $p\in\mathring{\M}$. Suppose that $\mathcal{A}:=\mathcal{A}(\boldsymbol{\mathfrak{a}})$ is a linear differential operator on $\M$ of order $l$ satisfying
\begin{align*}
\boldsymbol{\mathfrak{a}}=(a^r)_r\in \prod_{r=0}^l bc^{s}(\M,V^{\sigma+\tau+r}_{\tau+\sigma})\cap C^{n+k}(\U,V^{\sigma+\tau+r}_{\tau+\sigma}),\hspace{.5em}\text{ for }s\in[0,n],
\end{align*}
where $\U$ is an open subset containing $p$. 
Here $n+\infty:=\infty$, $n+\omega:=\omega$. Then for sufficiently small $\varepsilon_0$ and $r$
\begin{align*}
[\mu\mapsto{\tu}\mathcal{A}{T}^{-1}_{\mu}]\in{C^{k}({\B},C(I,\mathcal{L}({bc}^{s+l}(\M,V),{bc}^{s}(\M,V))))}.
\end{align*}
\end{prop}
\begin{remark}
The conditions in Proposition~\ref{S3: regularity of differential operators} can be equivalently stated as 
\begin{align*}
a_{\alpha}^{\kappa_{\sf p}}\in C^{n+k}(\varphi_{\kappa_{\sf p}}({\U}\cap\Ukp),\L(E)),\hspace*{1em}  a_{\alpha}^{\kappa}\in l_{\infty,\uf}(\boldsymbol{bc}^{s}(\mathbb{B},\L(E))),
\quad |\alpha|\le l,
\end{align*}
for the local expressions $\mathcal{A}_{\kappa}(x,\partial):=\sum\limits_{|\alpha|\leq{l}}a^{\kappa}_{\alpha}(x)\partial^{\alpha}$.
\end{remark}
\goodbreak
\begin{prop}
\label{S3: regularity of transformed functions}
Let $l\in\N_0$ and $k\in{\N}_{0}\cup\{\infty,\omega\}$. Suppose that 
\begin{center}
$u\in{C^{l+k}({\U},V)\cap{\F}^{s}(\M,V)}$, 
\end{center}
where either $s\in[0,l]$ if ${\F}=bc$, or $s=l$ if $\F=BC$. $\U$ has the same meaning as in Proposition~\ref{S3: regularity of differential operators}. Then for sufficiently small $\varepsilon_0$ and $r$, we have
\begin{center}
$[\mu\mapsto{\tu}u]\in{C^{k}({\B},C(I,{\F}^{s}(\M,V)))}$.
\end{center}
\end{prop}
\goodbreak
\begin{remark}
As a first application of the parameter-dependent diffeomorphism method and our main theorem, following the proof in Section~5, on a given closed $C^\omega$-uniformly regular Riemannian manifold $(\M,g)$, we can show that the solution to the heat equation
\begin{align*}
\begin{cases}
\partial_t u=\Delta_g u,\\
u(0)=u_0,
\end{cases}
\end{align*}
with $\Delta_g$ the Laplace-Beltrami operator with respect to the metric $g$, 
immediately becomes analytic jointly in time and space
for any initial value $u_0\in bc^s({\sf M})$, $s>0$.
\end{remark}
\bigskip
\section{\bf The Yamabe Flow}

A well-known problem in differential geometry is the Yamabe problem. In 1960, H.~Yamabe \cite{Yama60} conjectured the following:
\smallskip\\
\textbf{Yamabe Conjecture:} \emph{Let $(M,g)$ be a compact Riemannian manifold of dimension $m\geq 3$. Then every conformal class of metrics contains a representative with constant scalar curvature.}
\smallskip\\
The proof for Yamabe's conjecture was completed by N.S.~Trudinger \cite{Trud68}, T.~Aubin \cite{Aub82}, R.~Schoen \cite{Schoen84} using the calculus of variations and elliptic partial differential equations, see \cite{LeePar87} for a survey. The normalized Yamabe flow can be considered as another approach to this problem, which asks whether a metric converges conformally to one with constant scalar curvature on a compact manifold under the following flow:
\begin{align}
\label{S1: NYF}
\begin{cases}
\partial_t g=(s_g-R_g)g,\\
g(0)=g_0,
\end{cases}
\end{align}
where $R_g$ is the scalar curvature with respect to the metric $g$, and $s_g$ is the average of the
scalar curvature. It was introduced by R.~Hamilton shortly after the Ricci flow, and studied extensively by many authors afterwards, among them R.~Hamilton \cite{Ham89}, B.~Chow \cite{Chow92} and R.G.~Ye \cite{Ye94}, H.~Schwetlick and M.~Struwe \cite{SchStru03}, S.~Brendle \cite{Brend02, Brend05, Brend07}. A global existence and regularity result was presented by R.G.~Ye in \cite{Ye94}. The author asserts that the unique solution to \eqref{S1: NYF} exists globally and smoothly for any smooth initial metric. R.~Hamilton conjectured that on a compact Riemannian manifold the solution of \eqref{S1: NYF} converges to a metric of constant scalar curvature as $t\rightarrow\infty$. B.~Chow \cite{Chow92} commenced the study of Hamilton's conjecture and proved convergence in the case when $(M,g_0)$ is locally conformally flat and has positive Ricci curvature. Later, this result was improved by R.G.~Ye \cite{Ye94}, wherein the author removed the restriction on the positivity of Ricci curvature by lifting the flow to a sphere, and deriving a Harnack inequality. In the case that $3\leq m \leq 5$, H.~Schwetlick and M.~Struwe \cite{SchStru03} showed that the normalized Yamabe flow evolves any initial metric to one with constant scalar curvature as long as the initial Yamabe energy is small enough. In \cite{Brend05}, S.~Brendle was able to remove the smallness assumption on the initial Yamabe energy. A convergence result is stated in \cite{Brend07} by the same author for higher dimension cases. Finally, it is worthwhile to point out that for $2$-dimensional manifolds the Yamabe flow agrees with the Ricci flow.

As claimed in the introductory section, we will establish a regularity result for \eqref{S1: NYF} to show that for any initial metric $g_0$ belonging to the class $C^s$ with $s>0$ in a conformal class containing at least one real analytic metric, the solution to \eqref{S1: NYF} immediately evolves into an analytic metric. We shall point out here that not every conformal class contains a real analytic metric, but these conformal classes are in no sense trivial nevertheless. For instance, if $g$ is a real analytic metric and $f\in C^{\omega}(\M)$ has sufficiently small gradient, then we can also construct another real analytic metric by $g_f=g+\nabla f\otimes \nabla f$. In general, $g_f\notin [g]$. 
\smallskip\\
In section 5.1, we first study a generalization of the (unnormalized) Yamabe flow on non-compact, or even non-complete, manifolds. In recent years, these problems have been studied by several mathematicians, including Y.~An and L.~Ma \cite{AnMa99}, A.~Burchard, R.J.~McCann and A.~Smith \cite{Burch08}. In comparison with the existing results, we do not ask for a uniform bound on the curvatures of the initial metric and the background manifolds are admitted to be non-complete within certain conformal classes, see \cite[Section~1]{Ama13, Ama12} for justification. 

\subsection{\bf The Yamabe Flow on Uniformly Regular Manifolds without Boundary}
Let $(\M,g_0)$ be a $C^{\omega}$-uniformly regular Riemannian manifold without boundary of dimension $m \geq 3$. The Yamabe flow reads as
\begin{align}
\label{Yamabe flow eq1}
\begin{cases}
\partial_t g=-R_g g,\\
g(0)=g^0,
\end{cases}
\end{align}
where $R_g$ is the scalar curvature with respect to the metric $g$, and $g^0\in [g_0]$, the conformal class of the real analytic background metric $g_0$ of $\M$.  Note that on a compact manifold $\M$, \eqref{Yamabe flow eq1} is equivalent to \eqref{S1: NYF} in the sense that any solution to \eqref{Yamabe flow eq1} can be transformed into a solution to \eqref{S1: NYF} by a rescaling procedure.

We seek solutions to the Yamabe flow \eqref{Yamabe flow eq1} in $[g_0]$. Define $c(m):=\frac{m-2}{4(m-1)}$ and the Conformal Laplacian operator $L_g$ with respect to the metric $g$ as:
\begin{align*}
L_gu:=\Delta_g u -c(m)R_g u.
\end{align*}
Let $g=u^{\frac{4}{m-2}}g_0$ for some $u>0$. It is well known that
\begin{align*}
R_g=-\frac{1}{c(m)}u^{-\frac{m+2}{m-2}}L_0 u.
\end{align*}
where $L_0:=L_{g_0}$, see \cite[formula~(1.5)]{Ye94}. By rescaling the time variable, equation \eqref{Yamabe flow eq1} is equivalent to 
\begin{align}
\label{Yamabe flow eq2}
\begin{cases}
\partial_t u=u^{-\frac{4}{m-2}}L_0 u,\\
u(0)=u_0
\end{cases}
\end{align}
with a positive function $u_0$, see also \cite[formula~(7)]{MaCheng12}. In the following, we will use the continuous maximal regularity theory in Section~3 to establish existence and uniqueness of solutions to \eqref{Yamabe flow eq2}. (R4), Theorem \ref{retraction of BCk} and Remark~\ref{section 10: rmk}(b) imply that
\begin{align}
\label{section 8:regularity of g*}
g_0\in BC^{\infty}(\M,V^0_2)\cap C^{\omega}(\M,V^0_2),\hspace*{1em} g^{\ast}_0\in BC^{\infty}(\M,V^2_0)\cap C^{\omega}(\M,V^2_0).
\end{align}
By the well-known formula of scalar curvature in local coordinates,
\begin{align*}
R_{g}=\frac{1}{2}g^{ki}g^{lj}(g_{jk,li}+g_{il,kj}-g_{jl,ki}-g_{ik,lj}),
\end{align*}
see \cite[formula~(3.3.15), Definition~3.3.3]{Jost05}. This, combined with \eqref{section 8:regularity of g*}, thus yields
\begin{align}
\label{section 8:regularity of R_0}
R_{g_0}\in BC^{\infty}(\M)\cap C^{\omega}(\M).
\end{align}
Put 
\begin{center}
$P(u)h:=-u^{-\frac{4}{m-2}}\Delta_{g_0}h $, \hspace{.5em}and\hspace{.5em} $Q(u):=-c(m) u^{\frac{m-6}{m-2}}R_{g_0} $.
\end{center}
\smallskip 
Given any $0<s<1$, we choose $0<\alpha<s$, $\gamma=\frac{s-\alpha}{2}$, and $b>0$. Let $V=\K$ and
\begin{center}
$E_0:=bc^{\alpha}(\M)$, \hspace{.5em} $E_1:=bc^{2+\alpha}(\M)$, \hspace{.5em}and\hspace{.5em} $E_{\gamma}:=(E_0,E_1)^{0}_{\gamma,\infty}$.
\end{center}
Proposition~\ref{interpolation of LH spaces} implies that $E_{\gamma}=bc^s(\M)$. We put $W^{s}_{b}:=\{u\in bc^s(\M):\inf u>b\}$. By Proposition \ref{section 10: main prop}, for each $\beta\in\R$, we have
\begin{align}
\label{section 8: RA of unwgt LH}
\mathsf{P}_{\beta}: [u\mapsto u^{\beta}]\in C^{\omega}(W^{s}_{b},bc^{s}(\M)),
\end{align}
which together with \eqref{section 8:regularity of R_0} and Proposition \ref{pointwise multiplication properties} now shows that 
\begin{align}
\label{section 8: regularity of Q(u)}
Q(u)=-c(m) u^{-\frac{m-6}{m-2}} R_{g_0}\in E_0, \hspace*{1em} u\in W^{s}_{b}. 
\end{align}
In every local chart, the Laplace-Beltrami operator $\Delta_{g_0}$ reads as
\begin{align*}
\Delta_{g_0}=g_0^{jk}(\partial_j\partial_k-\Gamma^i_{jk}\partial_i),
\end{align*}
where $\Gamma^i_{jk}$ are the Christoffel symbols of $g_0$. By \eqref{section 8:regularity of g*} and Corollary \ref{differential operator properties}, we obtain
\begin{align}
\label{section 8: rho+lap}
\Delta_{g_0}\in \L(bc^{2+\alpha}(\M,V),bc^{\alpha}(\M,V)).
\end{align}
Then it follows from Proposition \ref{pointwise multiplication properties}, \eqref{section 8: RA of unwgt LH}-\eqref{section 8: rho+lap} that
\begin{center}
$(P,Q)\in C^{\omega}(W^{s}_{b},\L(E_1,E_0)\times E_0)$.
\end{center}
One verifies that the symbols of the principal parts for the local expressions $(P(u))_{\kappa}$ in every local chart $(\Uk,\varphi_{\kappa})$ satisfy
\begin{align*}
\hat{\sigma}(P(u))_{\kappa}^{\pi}(x,\xi)=({\kfk}u)^{-\frac{4}{m-2}}(x)g^{ij}(x)\xi_i \xi_j\geq C, \hspace*{1em}u\in W^{s}_{b}
\end{align*}
for any $\xi\in \R^m$ with $|\xi|=1$, and some $C>0$ independent of $\kappa$. It follows from Theorem \ref{section 6: H-MR} that
\begin{align}
\label{section 8.1 MR}
P(u)\in \mathcal{M}_{\gamma}(E_1,E_0), \hspace*{1em}u\in W^{s}_{b}.
\end{align}
Now \cite[Theorem~4.1]{CleSim01} leads to
\begin{theorem}
\label{section 8: well-posedness}
If $u_0\in W^{s}_{b}:=\{u\in bc^s(\M):\inf u>b\}$ with $0<s<1$ and $b>0$, then for every fixed $\alpha\in (0,s)$ equation \eqref{Yamabe flow eq2} has a unique local positive solution 
\begin{align*}
\hat{u}\in C^1_{1-\gamma}(J(u_0),bc^{\alpha}(\M))\cap{C_{1-\gamma}(J(u_0),bc^{2+\alpha}(\M))}\cap{C(J(u_0),W^{s}_b)}
\end{align*}
existing on $J(u_0):=[0,T(u_0))$ for some $T(u_0)>0$ with $\gamma=\frac{s-\alpha}{2}$.
\end{theorem}
\smallskip 
From now on, we use the notation $\hat{u}$ exclusively for the solution in Theorem \ref{section 8: well-posedness}. 
Next, set $G(u):=P(u)u-Q(u)$, $u\in W^s_b\cap E_1$. For any $(t_0,p)\in \dot{J}(u_0)\times\M$, we find a closed interval $I:=[\varepsilon,T]\subset \dot{J}(u_0)$ and $t_0\in\mathring{I}$. We define two function spaces as follows:
\begin{center}
$\ez(I):=C(I,E_0)$, \hspace{1em} $\ef(I):=C(I,E_1)\cap C^1(I,E_0)$,  
\end{center}
and an open subset of ${\ef}(I)$ by 
\begin{center}
$\mathbb{U}(I):=\{u\in\ef(I): \inf u>b\}$.
\end{center}
Let $u:={\ttl}\hat{u}=\hat{u}_{\lambda,\mu}$. By Proposition \ref{S3: time derivatives}, one computes that
\begin{align*}
u_t=\partial_t[{\rho}_{\lambda,\mu}]&=(1+\xi^{\prime}\lambda){\ttl}\hat{u}_t+{B}_{\lambda,\mu}(u)\\
&=-(1+\xi^{\prime}\lambda){\ttl}G(\hat{u})+{B}_{\lambda,\mu}(u)\\
&=-(1+\xi^{\prime}\lambda){\tu}G({\rh}\hat{u})+{B}_{\lambda,\mu}(u)\\
&=-(1+\xi^{\prime}\lambda){\tu}G({T}^{-1}_{\mu}u)+{B}_{\lambda,\mu}(u).
\end{align*}
Thus we define a map $\Phi: \mathbb{U}(I)\times\B\rightarrow \ez(I)\times E_1$ by
\begin{align*}
\Phi(v,(\lambda,\mu))=\binom{\partial_t v +(1+\xi^{\prime}\lambda){\tu}G(T^{-1}_{\mu}u)-B_{\lambda,\mu}(u)}{\gamma_{\varepsilon}v-\hat{u}(\varepsilon)},
\end{align*}
where $\gamma_{\varepsilon}$ is the evaluation map at $\varepsilon$, i.e., $\gamma_{\varepsilon}(v)=v(\varepsilon)$. Now the subsequent step is to verify the conditions of the implicit function theorem.
\smallskip\\
(i) Clearly, $\Phi(\hat{u}_{\lambda,\mu},(\lambda,\mu))=\binom{0}{0}$ for any $(\lambda,\mu)\in\B$. On the other hand, we have
\begin{align*}
D_1\Phi(\hat{u},(0,0))w=\binom{\partial_t w +DG(\hat{u})w}{\gamma_{\varepsilon}w}.
\end{align*}
One verifies that the symbol of the principal part for $DG(\hat{u})$ agrees with that of $P(\hat{u})$. By Theorem \ref{section 6: H-MR}, in \eqref{section 8.1 MR} we have the freedom to choose $\gamma$. Thus choose $\gamma=1$. It yields
\begin{align*}
D_1\Phi(\hat{u}(t,\cdot),(0,0))\in{\Lis}(\ef(I),\ez(I)\times E_1), \text{\hspace{1em}for every }t\in I.
\end{align*}
Set $\mathcal{A}(t):=DG(\hat{u}(t))$. The above formula shows that
\begin{center}
$(\frac{d}{dt}+\mathcal{A}(s),\gamma_{\varepsilon})\in{\Lis}({\ef}(I),{\ez}(I)\times{E_1})$,\hspace{.5em} for every $s\in{I}$.
\end{center}
It follows from \cite[Lemma~2.8(a)]{CleSim01} that 
\begin{align*}
D_1\Phi(\hat{u},(0,0))=(\frac{d}{dt}+\mathcal{A}(\cdot),\gamma_{\varepsilon}) \in {\Lis}(\ef(I), \ez(I)\times E_1).
\end{align*}

(ii) The next goal is to show that 
\begin{align}
\label{S5: goal 2}
\Phi\in C^{\omega}(\mathbb{U}(I)\times\B, \ez(I)\times E_1).
\end{align}
By Proposition \ref{S3: time derivatives}, ${B}_{\lambda,\mu}\in{C^{\omega}({\B},C(I,\mathcal{L}(E_1,E_0)))}$. We define a bilinear and continuous map $f:C(I,\mathcal{L}(E_1,E_0))\times{{\ef}(I)}\rightarrow{\ez}(I)$ by:
\begin{center}
$f: (T(t),u(t))\mapsto T(t)u(t)$.
\end{center}
Hence $[(v,(\lambda,\mu))\mapsto{f({B}_{\lambda,\mu},v)}={B}_{\lambda,\mu}(v)]\in{C^{\omega}(\mathbb{U}(I)\times{\B},{\ez}(I))}$.
\smallskip\\
By Proposition \ref{section 10: main prop 2}, the point-wise extension of $\mathsf{P}_{\beta}$ on $I$, i.e., $\mathsf{P}_{\beta}(u)(t)=\mathsf{P}_{\beta}(u(t))$, fulfils
\begin{align}
\label{section 8: RA of unwgt LH-time}
\mathsf{P}_{\beta}\in C^{\omega}(C(I,W^s_b),C(I,E_0)).
\end{align}
In virtue of Proposition \ref{S3: regularity of transformed functions}, we get
\begin{center}
$[\mu\mapsto {\tu}R_{g_0}]\in C^{\omega}(\B, \ez(I))$. 
\end{center}
By noticing ${\tu}Q(T_{\mu}^{-1}u)=-c(m)u^{-\frac{m-6}{m-2}}{\tu}R_{g_0}$, Proposition \ref{pointwise multiplication properties} and \eqref{section 8: RA of unwgt LH-time} hence imply that 
\begin{align*}
[(u,\mu)\mapsto {\tu}Q(T_{\mu}^{-1}u)]\in C^{\omega}(\mathbb{U}(I)\times\B,{\ez}(I)).
\end{align*}
On the other hand, in the light of Proposition \ref{S3: regularity of differential operators}, we infer that
\begin{center}
$[\mu\mapsto {\tu}\Delta_{g_0}{T_{\mu}^{-1}}]\in C^{\omega}(\B,C(I,\L(E_1,E_0)))$.
\end{center}
It follows again from Proposition \ref{pointwise multiplication properties} and \eqref{section 8: RA of unwgt LH-time} 
that
\begin{align*}
[(u,\mu)\mapsto {\tu}P(T_{\mu}^{-1}u)] \in C^{\omega}(\mathbb{U}(I)\times\B,{\ez}(I)).
\end{align*}
The rest of proof for \eqref{S5: goal 2} follows straight away. Consequently, we have proved the desired assertion. Now we are in a position to apply the implicit function theorem. Hence there exists an open neighborhood, say $\mathbb{B}(0,r_0)\subset{\B}$, such that
\begin{align*}
[(\lambda,\mu)\mapsto{\hat{u}}_{\lambda,\mu}]\in{C^{\omega}(\mathbb{B}(0,r_0),{\ef}(I))}.
\end{align*}
As a consequence of Theorem \ref{S3: inverse}, we deduce that 
\begin{theorem}
For each $u_0\in W^{s}_{b}:=\{u\in bc^s(\M):\inf u>b\}$ with $0<s<1$ and $b>0$, equation \eqref{Yamabe flow eq2} has a unique local positive solution 
\begin{align*}
\hat{u}\in C^{\omega}(\dot{J}(u_0)\times \M)
\end{align*}
for some interval $J(u_0):=[0,T(u_0))$. In particular, 
\begin{align*}
g=\hat{u}^{\frac{4}{m-2}}g_0\in C^{\omega}(\dot{J}(u_0)\times \M,V^0_2)
\end{align*}
\end{theorem}
\begin{rmk}
\begin{itemize}
\item[]{\phantom{ some  text to complete some  } }
\item[(a)] In general, the presence of a metric $g_0\in BC^{\infty}(\M,V^0_2)\cap C^{\omega}(\M,V^0_2)$ is unnecessary. As long as there is a metric $g_0\in bc^{2+\alpha}(\M,V^0_2)\cap C^{k}(\M,V^0_2)$ with $k\in\{\infty,\omega\}$, the solution to \eqref{Yamabe flow eq1} on $\M$ immediately evolves conformally in $[g_0]$ into the class $C^{k}(\dot{J}(u_0)\times \M,V^0_2)$.
\item[(b)]  For any  Riemannian manifold $\M$ without boundary, by \cite[Theorem~1.4, Corollary~1.5]{MullerAx}, in every conformal class containing a $C^k$-metric with $k\in\{\infty,\omega\}$ there exists a $C^k$-metric $g_0$ with bounded geometry, which makes $(\M,g_0)$ complete. It then follows that $(\M,g_0)$ is uniformly regular. We can infer from the proof above that the solution to \eqref{Yamabe flow eq1} belongs to $C^{k}(\dot{J}(u_0)\times \M,V^0_2)$.
\end{itemize}
\end{rmk}
\medskip

\subsection{\bf The Normalized Yamabe Flow on Compact Manifolds}
Next, we consider regularity of the normalized Yamabe flow on compact manifolds. Let $(\M,g_0)$ be a connected compact closed $m$-dimensional $C^{\infty}$-Riemannian manifold with $m\geq 3$.  By a well-known result of H. Whitney, $\M$ admits a compatible real analytic structure. The existence of a real analytic metric on $\M$ follows from \cite{Morrey58}. Without loss of generality, we may assume that the atlas $\mathfrak{A}$ and the metric $g_0$ are real analytic.

The normalized Yamabe flow reads as
\begin{align}
\label{Normalized Yamabe flow eq1}
\begin{cases}
\partial_t g=(s_g-R_g)g,\\
g(0)=g^0,
\end{cases}
\end{align}
where $g^0\in[g_0]$, $V(g)=\int\limits_{\M} dV_g$ and 
\begin{center}
$s_g=\frac{1}{V(g)}\int\limits_{\M} R_g\, dV_g$.
\end{center}
The normalized Yamabe flow preserves the volume, that is, $V(g)\equiv V(g^0)$. This can be easily verified by checking the time derivative of $V(g)$.
\smallskip\\
As in the above subsection, we can reduce equation \eqref{Normalized Yamabe flow eq1} to the following form. 
\begin{align}
\label{Normalized Yamabe flow eq3}
\begin{cases}
\partial_t u=u^{-\frac{4}{m-2}}[L_0 u+c(m)s_g u^{\frac{m+2}{m-2}}],\\
u(0)=u_0.
\end{cases}
\end{align}
Put 
\begin{center}
$P(u)h:=-u^{-\frac{4}{m-2}}[\Delta_{g_0}h - \frac{u^{\frac{m+2}{m-2}}}{V(g)}\int\limits_{\M} u^{-\frac{m+2}{m-2}}L_0 h\, dV_g]$,
\end{center}
and
\begin{center}
$Q(u)=-c(m)u^{\frac{m-6}{m-2}}R_{g_0} $.
\end{center}
\smallskip 
Given $0<s<1$, we choose $0<\alpha<s$ and $\gamma=\frac{s-\alpha}{2}$. Set $E_0$, $E_1$ and $E_{\gamma}$ as in Section~5.1. Additionally, we put $W^{s}:=\{u\in bc^s(\M):u>0\}$. Let
\begin{align}
\label{section 8 eq-A(u)}
A(u)h:=\frac{1}{V(g)}\int\limits_{\M} u^{-\frac{m+2}{m-2}}L_0 h\, dV_g=\frac{1}{\int\limits_{\M} u^{\frac{2m}{m-2}}\, dV_{g_0}}\int\limits_{\M} u L_0 h\, dV_{g_0}.
\end{align}
By Corollary~\ref{differential operator properties} and Proposition~\ref{section 10: main prop}, for each $\beta\in\R$
\begin{align}
\label{S8: regularity of Lap & power}
\Delta_{g_0}\in\L(BC^2(\M),BC(\M)),\hspace*{.5em}\text{ and }\hspace*{.5em}\mathsf{P}_{\beta}\in C^{\omega}(W^{s},bc^s(\M)).
\end{align}
One readily checks that \eqref{section 8 eq-A(u)}, \eqref{S8: regularity of Lap & power} and Remark~\ref{section 10: rmk}(a) imply that
\begin{align*}
A \in C^{\omega}(W^{s},\L(BC^2(\M),\R)).
\end{align*}
It yields
\begin{center}
$(P,Q)\in C^{\omega}(W^{s},\L(E_1,E_0)\times E_0)$.
\end{center}
We verify that the symbol of the principal part for the local expression of $u^{-\frac{4}{m-2}}\Delta_{g_0}$ in every local chart $(\Uk,\varphi_{\kappa})$ satisfies
\begin{align*}
\hat{\sigma}(P(u))_{\kappa}^{\pi}(x,\xi)=u^{-\frac{4}{m-2}}(x)g^{ij}(x)\xi^i \xi^j \geq C, \hspace*{1em}u\in W^{s}
\end{align*}
for any $\xi\in \R^m$ with $|\xi|=1$, and some $C>0$ by the compactness of $\M$. Thus $P(u)$ is normally elliptic for any $u\in{W^{s}}$. As a conclusion of Theorem~\ref{section 6: H-MR}, we have
\begin{align}
\label{section 8.2 MR}
P(u)\in \mathcal{M}_{\gamma}(E_1,E_0),\hspace*{1em} u\in W^{s},
\end{align}
since $ u^{-\frac{m+2}{m-2}}A(u)$ is a lower order perturbation compared to $u^{-\frac{4}{m-2}}\Delta_{g_0}$ for every $u\in W^s$. It follows from \cite[Lemma~2.7, Theorem~4.1]{CleSim01} that
\begin{theorem}
If $u_0\in W^{s}:=\{u\in bc^s(\M):\inf u>0\}$ with $0<s<1$, then for each fixed	$\alpha\in (0,s)$ equation \eqref{Normalized Yamabe flow eq3} has a unique local positive solution 
\begin{align*}
\hat{u}\in C^1_{1-\gamma}(J(u_0),bc^{\alpha}(\M))\cap{C_{1-\gamma}(J(u_0),bc^{2+\alpha}(\M))}\cap{C(J(u_0),W^{s})}
\end{align*}
existing on $J(u_0):=[0,T(u_0))$ for some $T(u_0)>0$ with $\gamma=\frac{s-\alpha}{2}$.
\end{theorem}

Next, set $G(u)=P(u)u-Q(u)$,  $u\in W^s\cap E_1$. 
For any closed interval $I:=[\varepsilon,T]\subset \dot{J}(u_0)$, the function spaces $\ez(I)$, $\ef(I)$ and $\mathbb{U}(I)$ have the same definitions as in Section~5.1 with $b=0$.
The map $\Phi: \mathbb{U}(I)\times\B\rightarrow \ez(I)\times E_1$ is defined by
\begin{align*}
\Phi(v,(\lambda,\mu))=\binom{\partial_t v +(1+\xi^{\prime}\lambda){\tu}G(T^{-1}_{\mu}u)-B_{\lambda,\mu}(u)}{\gamma_{\varepsilon}v-u(\varepsilon)}.
\end{align*}
Clearly, $\Phi(\hat{u}_{\lambda,\mu},(\lambda,\mu))=\binom{0}{0}$ for any $(\lambda,\mu)\in\B$. On the other hand, we have
\begin{align*}
D_1\Phi (\hat{u},(0,0))w=\binom{\partial_t w +DG(\hat{u})w}{\gamma_{\varepsilon}w}.
\end{align*}
Since $DA(\hat{u})\in \L(BC^2(\M),\R)$, it follows from a similar argument to the previous problem that
\begin{align*}
D_1(\hat{u},(0,0))\in {\Lis}(\mathbb{U}(I), \ez(I)\times E_1).
\end{align*}
Now it remains to show that $\Phi\in C^{\omega}(\mathbb{U}(I)\times\B, \ez(I)\times E_1)$. Following the previous proof, we only need to consider $\bar{A}(u):=A(u)u$. Note that
\begin{align*}
{\ttm}\bar{A}(\Theta^{\mu}_{\ast}u)=\frac{1}{\int\limits_{\M} (\Theta^{\mu}_{\ast}u)^{\frac{2m}{m-2}}\, dV_{g_0}}\int\limits_{\M} \Theta^{\mu}_{\ast}u L_0 \Theta^{\mu}_{\ast}u\, dV_{g_0}.
\end{align*}
An analogous argument as in \cite[Proposition~6.2]{ShaoPre} yields
\begin{align*}
{\ttm}\bar{A}(\Theta^{\mu}_{\ast}\cdot)\in C^{\omega}(\B,\L(E_1,\R)).
\end{align*} 
The rest follows by a slight change of step (ii) in Section~5.1 and \cite[Lemma~5.1]{EscPruSim031}. Then we adopt the implicit function theorem and get
\begin{theorem}
\label{S5: NYF-analyticity}
If $u_0\in W^{s}:=\{u\in bc^s(\M):\inf u>0\}$ with $0<s<1$, then equation \eqref{Normalized Yamabe flow eq3} has a unique local positive solution 
\begin{align*}
\hat{u}\in C^{\omega}(\dot{J}(u_0)\times \M)
\end{align*}
for some interval $J(u_0):=[0,T(u_0))$. In particular, 
\begin{align*}
g=\hat{u}^{\frac{4}{m-2}}g_0\in C^{\omega}(\dot{J}(u_0)\times \M,V^0_2)
\end{align*}
\end{theorem}
\begin{rmk}
\begin{itemize}
\item[]{\phantom{ some  text to complete some  } }
\item[(a)] We shall point out the global existence of the real analytic solution obtained above. By \cite[Theorem~4.1(d)]{CleSim01}, it suffices to show that 
${\rm{dist}}(u(t,\cdot),\partial W^s)>0$ and there exists a $\theta\in (\gamma,1]$ such that $\|\hat{u}\|_{E_{\theta}}$ is bounded in finite time. Actually, this is guaranteed by \cite[Theorem~3, Lemma~4]{Ye94}.
\item[(b)] It follows from Theorem~\ref{S5: NYF-analyticity} that the equilibria in the conformal class $[g_0]$, that is to say, those of constant scalar curvature, must be analytic. 
This implies that the solution to the Yamabe problem in any conformal class possessing at least one real analytic metric is analytic.
\end{itemize}
\end{rmk}
\bigskip

\section{\bf Appendix}

\begin{lem}
\label{S6:: sub-op LH}
Let $k\in\N$ and $s\in [0,k]$. Suppose that $f\in C^k(U,\R)$ for some open interval $U\subset \R$ and $u\in W:=\{u\in bc^{s}(\M):\im(u)\subset\subset U \}$. Then $f(u):=f\circ u \in bc^s(\M)$.
\end{lem}
\begin{proof}
It suffices to prove the assertion for $0<s<1$, since the other cases are direct consequences of the chain rule. Clearly, $f(u)\in BC(\M)$. On the other hand, for $|x-y|<\delta$ for some sufficiently small $\delta$
\begin{align*}
&\hspace*{1.2em}\frac{|{\kfk}\pi_{\kappa}(x){\kfk}f(u)(x)-{\kfk}\pi_{\kappa}(y){\kfk}f(u)(y)|}{|x-y|^s}\\
&={\kfk}\pi_{\kappa}(y)\frac{|{\kfk}f(u)(x)-{\kfk}f(u)(y)|}{|x-y|^s}+|{\kfk}f(u)(x)|\frac{|{\kfk}\pi_{\kappa}(x)-{\kfk}\pi_{\kappa}(y)|}{|x-y|^s}\\
&\leq{\kfk}\pi_{\kappa}(y) \int\limits_0^1 |f^{\prime}({\kfk}u(y)+t({\kfk}u(x)-{\kfk}u(y)))|\frac{|\zeta(x){\kfk}u(x)-\zeta(y){\kfk}u(y)|}{|x-y|^s}\, dt\\
&\hspace*{.5em}+\delta^{1-s}\|{\kfk}\pi_{\kappa}\|_{1,\infty,\Xk}\|f(u)\|_{\infty}\\
&\leq M(u)[\zeta{\kfk}u]^{\delta}_{s,\infty} +\delta^{1-s}\|{\kfk}\pi_{\kappa}\|_{1,\infty,\Xk}\|f(u)\|_{\infty}.
\end{align*}
The validity of the last step is supported by (L2). By \eqref{S2: infnty,uf} and \eqref{S3: pull onto Euc by zeta}, we attain  
\begin{align*}
\mathcal{R}^{c}f(u)\in l_{\infty,\uf}(\boldsymbol{bc}^s).
\end{align*} 
Now the assertion follows from Theorem \ref{retraction of BCk}.
\end{proof}
Henceforth, we assume that $\alpha\in\R$, $s\geq 0$ and $b>0$. 
\begin{lem}
\label{S6: lem}
Suppose that $W^{s}_b:=\{u\in bc^{s}(U):\inf u>b\}$ for some open subset $U\subset \R^m$. Then
\begin{align*}
[u\mapsto u^{\alpha}]\in C^{\omega}(W^{s}_b,bc^{s}(U)).
\end{align*}
\end{lem}
\begin{proof}
We only need to show the case $0<s<1$, and the rest follows from knowledge of calculus. Put $W:=\{u\in bc^{s}(U):\|u\|_{\infty}<1\}$. It is well known that 
\begin{align}
\label{S6: series}
(1+u)^{\alpha}=\sum\limits_{n} \binom{\alpha}{n}u^n,\hspace{1em} \text{for }u\in W.
\end{align}
converges in $BC(U)$. \eqref{section 2: Holder/bundle} implies that $\label{section 10: Holder norm}
[u^n]^{\delta}_{s,\infty}\leq n \|u\|_{\infty}^{n-1} [u]^{\delta}_{s,\infty}$. Hence,
\begin{align*}
\sum\limits_{n} \binom{\alpha}{n}\|u^n\|_{s,\infty}\leq \sum\limits_n \binom{\alpha}{n}\|u\|_{\infty}^n +\alpha [u]^{\delta}_{s,\infty}\sum\limits_n \binom{\alpha-1}{n}\|u\|_{\infty}^{n}<\infty.
\end{align*}
By Lemma~\ref{S6:: sub-op LH}, \eqref{S6: series} converges in $bc^s(U)$.
\smallskip\\
For each $u\in W^{s}_b$, choose $\varepsilon<\min\{b,\inf u-b \}$. Then for any $v\in \mathbb{B}_{\varepsilon}(u)\subset W^{s}_b$,
\begin{align*}
v^{\alpha}=u^{\alpha}(1+\frac{v-u}{u})^{\alpha}=u^{\alpha}\sum\limits_n \binom{\alpha}{n}(\frac{v-u}{u})^{n}=\sum\limits_n \binom{\alpha}{n}u^{(\alpha-n)}(v-u)^{n}.
\end{align*}
By Proposition~\ref{pointwise multiplication properties} and Lemma~\ref{S6:: sub-op LH}, the series converges in $bc^s(U)$.
\end{proof}
\begin{prop}
\label{section 10: main prop}
Suppose that $W^{s}_b:=\{u\in bc^{s}(\M):\inf u>b\}$. Then $\mathsf{P}_{\alpha}:=[u\mapsto u^{\alpha}]\in C^{\omega}(W^{s}_b, bc^{s}(\M))$.
\end{prop}
\begin{proof}
For each $u\in W^{s}_b$, choose $\varepsilon<\min\{b,\inf u-b \}$ sufficiently small. By Lemma~\ref{S6: lem}, \eqref{S3: pull onto Euc by zeta} and \eqref{S3: pointwise-mul-Rm}, for any $v\in \mathbb{B}_{\varepsilon}(u)\subset W^{s}_b$
\begin{align*}
\mathcal{R}_{\kappa}^{c}v^{\alpha}&={\kfk}\pi_{\kappa}(\zeta{\kfk}u)^{(\alpha-n)}\sum\limits_n \binom{\alpha}{n}(\zeta{\kfk}v-\zeta{\kfk}u)^{n}\\
&=\sum\limits_n \binom{\alpha}{n}{\kfk}\pi_{\kappa}({\kfk}u)^{(\alpha-n)}({\kfk}v-{\kfk}u)^{n}
\end{align*}
converges in $bc^{s}_{\kappa}$ uniformly with respect to $\kappa$. One can check that
\begin{align*}
\mathcal{R}^c\circ \mathsf{P}_{\alpha}\in C^{\omega}(W^{s}_b,l_{\infty,\uf}(\boldsymbol{bc}^s)).
\end{align*}
The the assertion follows from Theorem~\ref{retraction of BCk} and the fact that $\mathcal{R}$ is real analytic.\end{proof}
A slight modification of the above proofs now implies the point-wise extension of Proposition~\ref{section 10: main prop}.
\begin{prop}
\label{section 10: main prop 2}
Let $X:=bc^{s}(\M)$ and $W^{s}_b(I):=\{u\in C(I,X):\inf u>b \}$. Then $\mathsf{P}_{\alpha}\in C^{\omega}(W^{s}_b(I),C(I,X))$.
\end{prop}
\begin{rmk} 
\label{section 10: rmk} 
\begin{itemize}
\item[]{\phantom{ some  text to complete some  } }
\item[(a)] When $\M$ is compact, we can choose $b=0$ in Proposition~\ref{section 10: main prop} and \ref{section 10: main prop 2}. 
\item[(b)] The estimates in Lemma~\ref{S6:: sub-op LH} and \ref{S6: lem} show that the case $bc$ replaced by $BC$ is admissible. So Proposition~\ref{section 10: main prop} and \ref{section 10: main prop 2} still hold true for $BC^s(\M)$ and $C(I, BC^s(\M))$, respectively.
\end{itemize}
\end{rmk}
\section*{Acknowledgements}
The authors would like to thank Herbert Amann, Marcelo Disconzi and Marcus Khuri for helpful discussions and valuable suggestions. 


\begin{thebibliography}{99}
 
 \bibitem{Ama95} H. Amann, \emph{Linear and Quasilinear Parabolic Problems: Volume I. } Birkh\"auser Boston, Inc., Boston, MA, 1995.

 \bibitem{Ama01} H. Amann, \emph{Elliptic operators with infinite-dimensional state spaces.} J. Evol. Equ. \textbf{1}, no. 2, 143-188, (2001).

 \bibitem{Ama13} H. Amann, \emph{Function spaces on singular manifolds.}  Math. Nachr. \textbf{286}, No. 5-6, 436-475, (2013).

 \bibitem{Ama12} H. Amann, \emph{Anisotropic function spaces on singular manifolds.} 
 arXiv.1204.0606.

\bibitem{Ama13b} H. Amann, \emph{Parabolic equations on uniformly regular Riemannian manifolds and degenerate initial boundary value problems.} Preprint, (2013).

\bibitem{Ama14} H. Amann, \emph{Parabolic equations on singular manifolds.} In preparation, (2013).

\bibitem{AnMa99} Y. An, L. Ma, \emph{The maximum principle and the Yamabe flow}. ``Partial 
Differential Equations and Their Applications", World Scientific, Singapore, pp 211-224, 1999.

 \bibitem{Ange90} S.B. Angenent, \emph{Nonlinear analytic semiflows.} Proc. Roy. Soc. Edinburgh Sect. A \textbf{115}, no. 1-2, 91-107, (1990). 
 
 \bibitem{Aub82} T. Aubin, \emph{Nonlinear Analysis on Manifolds. Monge-Ampere Equations.} Springer-Verlag, New York, 1982.
 
 \bibitem{Band12} C. Bandle, F. Punzo, A. Tesei, \emph{Existence and nonexistence of patterns on Riemannian manifolds.} J. Math. Anal. Appl. \textbf{387}, no. 1, 33-47,  (2012).
 
 \bibitem{Brend02} S. Brendle, \emph{A generalization of the Yamabe flow for manifolds with boundary.} Asian J. Math. \textbf{6}, no. 4, 625-644, (2002). 
 
 \bibitem{Brend05} S. Brendle, \emph{Convergence of the Yamabe flow for arbitrary initial energy.} J. Differential Geom. \textbf{69}, no. 2, 217-278, (2005).
 
 \bibitem{Brend07} S. Brendle, \emph{Convergence of the Yamabe flow in dimension 6 and higher.} Invent. Math.
\textbf{170}, 541-576, (2007).
 
  \bibitem{Burch08} A. Burchard, R.J. McCann, A. Smith, \emph{Explicit Yamabe flow of an asymmetric cigar.} Methods Appl. Anal. \textbf{15}, no. 1, 65-80, (2008).
 
 \bibitem{Chow92} B. Chow, \emph{The Yamabe flow on locally conformally flat manifolds with positive Ricci curvature.} Comm. Pure Appl. Math. \textbf{45}, no. 8, 1003-1014, (1992).
 
 \bibitem{CleSim01} P. Cl\'ement, G. Simonett, \emph{Maximal regularity in continuous interpolation spaces and quasilinear parabolic equations.} J. Evol. Equ. \textbf{1}, no. 1, 39-67, (2001).
 
 \bibitem{Greene78} R.E. Greene, \emph{Complete metrics of bounded curvature on noncompact manifolds.} 
Arch. Math. (Basel) \textbf{31}, no. 1, 89-95, (1978/79).
 
 \bibitem{DaPra79} G. Da Prato, P. Grisvard, \emph{Equations d'\'evolution abstraites non lin\'eaires de type parabolique.} Ann. Mat. Pura Appl. (4) \textbf{120}, 329-396, (1979).  
 
 \bibitem{DaPra88} G. Da Prato, A. Lunardi, \emph{Stability, instability and center manifold theorem for fully nonlinear autonomous parabolic equations in Banach space}. Arch. Rational Mech. Anal. \textbf{101}, no. 2, 115-141, (1988).  
 
 \bibitem{EscSim98}  J. Escher, G. Simonett, \emph{The volume preserving mean curvature flow near spheres.} Proc. Amer. Math. Soc. \textbf{126}, no. 9, 2789-2796, (1998).
 
 \bibitem{EscMaySim98} J. Escher, U. Mayer, G. Simonett, \emph{The surface diffusion flow for immersed hypersurfaces.} SIAM J. Math. Anal. \textbf{29}, no. 6, 1419-1433, (1998).
 
 \bibitem{EscSim98} J. Escher, G. Simonett, \emph{A center manifold analysis for the Mullins-Sekerka Model.} J. Differential Equations \textbf{143}, no. 2, 267-292, (1998).
 
 \bibitem{EscPruSim031} J.~Escher, J.~Pr\"uss, G.~Simonett, \emph{A new approach to the regularity of solutions for   parabolic equations.} Evolution equations, 167-190, Lecture Notes in Pure and Appl. Math. \textbf{234}, Dekker, New York, (2003).
 
 \bibitem{EscPruSim032} J. Escher, J. Pr\"uss, G. Simonett, \emph{Analytic solutions for a Stefan problem with Gibbs-Thomson correction.} J. Reine Angew. Math. \textbf{563}, 1-52, (2003).
 
 \bibitem{Ham89} R. Hamilton, \emph{Lectures on geometric flows}. Unpublished manuscript (1989).
 
 \bibitem{Jost05} J. Jost, \emph{Riemannian Geometry and Geometric Analysis. } Fourth edition. Universitext. Springer-Verlag, Berlin, 2005.
 
 \bibitem{LeePar87} J. M. Lee, T.H. Parker, \emph{The Yamabe problem.} Bull. Amer. Math. Soc. (N.S.) \textbf{17}, no. 1, 37-91, (1987).
 
 \bibitem{Lunar95} A. Lunardi, \emph{Analytic Semigroups and Optimal Regularity in Parabolic Problems.} Birkh\"auser Verlag, Basel, 1995.
 
 \bibitem{MaCheng12} L. Ma, L. Cheng, A. Zhu, \emph{Extending Yamabe flow on complete Riemannian manifolds. Bull.} Sci. Math. \textbf{136}, no. 8, 882-891, (2012).
 
 \bibitem{Mazz06} A.L. Mazzucato, V. Nistor, \emph{Mapping properties of heat kernels, maximal regularity, and semi-linear parabolic equations on noncompact manifolds.} J. Hyperbolic Differ. Equ. \textbf{3}, no. 4, 599-629, (2006). 
 
 \bibitem{Morrey58} C.B. Morrey, Jr., \emph{The analytic embedding of abstract real-analytic manifolds.} Ann. of Math. (2) \textbf{68}, 159-201, (1958). 
 
 \bibitem{MullerAx} O. M\"uller, M. Nardmann, \emph{Every conformal class contains a metric of bounded geometry.} 	arXiv:1303.5957.
 
 \bibitem{Punzo12} F. Punzo, \emph{Blow-up of solutions to semilinear parabolic equations on Riemannian manifolds with negative sectional curvature.} J. Math. Anal. Appl. \textbf{387}, no. 2, 815-827, (2012). 
 
 \bibitem{Schoen84} R. Schoen, \emph{Conformal deformation of a Riemannian metric to constant scalar curvature.} J. Differential Geom. \textbf{20}, no. 2, 479-495, (1984).
 
 \bibitem{SchStru03} H. Schwetlick, M. Struwe, \emph{Convergence of the Yamabe flow for large energies}. J. Reine Angew. Math. \textbf{562}, 59-100, (2003).
 
 \bibitem{ShaoPre} Y. Shao, \emph{A family of parameter-dependent diffeomorphisms acting on function spaces over a Riemannian manifold and applications to geometric flows.} 
{arXiv}. 
 
 \bibitem{Trib78} H. Triebel, \emph{Interpolation Theory, Function Spaces, Differential Operator.} North-Holland Publishing Co., Amsterdam-New York, 1978.
 
 \bibitem{Trud68} N.S. Trudinger, \emph{Remarks concerning the conformal deformation of Riemannian structures on compact manifolds. } Ann. Scuola Norm. Sup. Pisa (3) \textbf{22}, 265-274, (1968). 

 \bibitem{Yama60} H. Yamabe, \emph{On a deformation of Riemannian structures on compact manifolds}. Osaka Math J. \textbf{12}, 21-37, (1960).
 
 \bibitem{Ye94} R.G, Ye, \emph{Global existence and convergence of Yamabe flow.} 
 J. Differential Geom. \textbf{39}, no. 1, 35-50, (1994).
 
 \bibitem{Zhang97} Qi S. Zhang, \emph{Nonlinear parabolic problems on manifolds, and a nonexistence result for the noncompact Yamabe problem.} Electron. Res. Announc. Amer. Math. Soc. \textbf{3}, 45-51 (electronic), (1997).
 
 \bibitem{Zhang00} Qi S. Zhang, \emph{Semilinear parabolic problems on manifolds and applications to the non-compact Yamabe problem.} Electron. J. Differential Equations \textbf{2000}, No. 46, 30 pp. (electronic).  
 \end{thebibliography}
 \end{document}